\documentclass{article}
\usepackage{latexsym, amscd, amsfonts, eucal, mathrsfs, amsmath, amssymb, amsthm, xypic,xr, makeidx, stmaryrd, color, enumerate, tikz}
\usepackage{appendix}
\usepackage{multicol}
\usepackage[all]{xy}
\usepackage{hyperref}
\newtheorem{theorem}{Theorem}[section]
\newtheorem*{thm}{Theorem}
\newtheorem{lemma}[theorem]{Lemma}
\newtheorem{proposition}[theorem]{Proposition}
\newtheorem{corollary}[theorem]{Corollary}

\theoremstyle{definition}
\newtheorem{definition}[theorem]{Definition}
\newtheorem{construction}[theorem]{Construction}
\newtheorem{convention}[theorem]{Convention}

\newtheorem{notation}[theorem]{Notation}
\newtheorem{warning}[theorem]{Warning}
\newtheorem{remark}[theorem]{Remark}
\newtheorem*{question}{Question}

\usepackage{cleveref}
\usepackage{scrextend}

\newcommand{\hf}{\textup{H}\underline{\mathbb{F}}}
\newcommand{\rpi}{\underline{\pi}}
\newcommand{\fpi}{\pi^{C_2}}
\newcommand{\rpis}{\rpi_{\star}}
\newcommand{\F}{\mathbb{F}}
\newcommand{\uF}{\underline{\mathbb{F}}}
\newcommand{\bpr}{\textnormal{BP}\mathbf{R}}
\newcommand{\emu}{\textnormal{E}\mu_2}

\newcommand{\hz}{\textup{H}\underline{\mathbb{Z}}}

\newcommand{\mur}{\textup{MU}\mathbf{R}}
\newcommand{\ev}{\textup{Even}(S^0)}
\newcommand{\arep}{\mathcal{A}_{\star}}
\newcommand{\rarep}{\overline{\mathcal{A}}_{\star}}

\DeclareMathOperator*{\hocolim}{hocolim}
\DeclareMathOperator*{\holim}{holim}
\DeclareMathOperator*{\colim}{colim}

\usepackage{filecontents}

\usepackage[citestyle=alphabetic,bibstyle=alphabetic,backend=bibtex]{biblatex}
\renewbibmacro{in:}{}
\bibliography{statement}

\begin{document}
\title{Power operations for $\hf_2$ and a cellular construction of $\bpr$}
\author{Dylan Wilson}
\maketitle
\begin{abstract} We study some power operations for ordinary $C_2$-equivariant
homology with coefficients in the constant Mackey functor $\underline{\F}_2$. In addition to a few foundational
results, we calculate the action of these power operations on a $C_2$-equivariant dual Steenrod
algebra. As an application, we give a cellular construction of the $C_2$-spectrum $\bpr$ and deduce
its slice tower.
\end{abstract}
\tableofcontents
\newpage

\section*{Introduction}
From a user's perspective, the primary goal of this paper is to construct equivariant
power operations for $\hf_2$ and compute their action on an equivariant 
dual Steenrod algebra.
This comprises the bulk of the work below, and the
main results are contained in Theorem \ref{omnibus} and Theorem \ref{action} below. They
can be summarized as follows:

\begin{thm} If $X$ is a $C_2$-spectrum admitting an equivariant symmetric multiplication,
$\mathbb{P}_2(X) \to X$, then $(\hf_2)_{\star}X$ admits natural operations $Q^{n\rho}$
of degree $n\rho$
for $n \in \mathbb{Z}$
and satisfying analogues of the standard properties of the
non-equivariant Dyer-Lashof operations as in, e.g.,  \cite[Ch. III]{Hinfty}.
\end{thm}

\begin{thm} As an algebra over the ring of equivariant power operations, the 
equivariant dual Steenrod algebra $\rpis(\hf_2 \wedge \hf_2)$ is generated by
$\tau_0 \in \pi^{C_2}_1(\hf_2 \wedge \hf_2)$. 
\end{thm}

There are several departures from the classical story in the proofs. Any
non-equivariant result that depends on the fact that $\pi_*\textup{H}\mathbb{F}_2$ vanishes away
from $\pi_0$ is now much more delicate, since the equivariant homology of a point
is far from trivial. For example, in general there is no Thom isomorphism which computes
the homology of extended powers of representation spheres.

Nevertheless, it is possible to give a complete analysis of the extended powers of \emph{regular}
representation spheres, $S^{n\rho}$,
and single desuspensions of these- in other words, slice cells. This aligns
with the emerging philosophy on equivariant homotopy theory:
\newline\newline
{\bf The Slice Filtration}: It's ordinary homotopy theory, only twisted.\footnote{This 
slogan is inspired by the 1997 ad campaign for the
orange-flavored soft drink `Slice'.}
\newline

That said, we would still like to have operations for classes in degree $V$ for an
arbitrary representation sphere. After all, the purported generator
$\tau_0$ in the dual Steenrod algebra
is in degree 1, which is not of the form $n\rho -1$. To do this, we use a trick inspired
by Goodwillie calculus. The idea is to replace the extended powers $\mathbb{P}_2(S^{n\rho})$
by the co-linearization:
	\[
	\textup{Ops}_{\hf_2} := \holim_n (\hf_2 \wedge \Sigma^{n\rho}\mathbb{P}_2(S^{-n\rho})).
	\]
Since regular representation spheres are cofinal amongst all representation spheres, this allows
us to define power operations in general degrees. This trick
has a long history, and can be found in various degrees of overtness in the
references
\cite{JW, Lur, Lur2, KM}. 

The use of this spectrum $\textup{Ops}_{\hf_2}$ has more advantages. In \S4 we identity
this spectrum with the desuspension of a Tate construction, and this allows us to easily
deduce the Steenrod coaction on the homotopy of this spectrum. In \S5 we apply this calculation
to prove the Nishida relations and compute the action of power operations on
the dual Steenrod algebra. This approach is worthwhile even in the non-equivariant
case, where it gives a clean derivation of Steinberger's formulas in \cite[III.2]{Hinfty}.
\newline

Before describing the second main goal of this paper, it is perhaps helpful to give some
motivation.

In the course of their resolution of the Kervaire invariant problem, Hill, Hopkins, and Ravenel
used a machine-crafted $C_8$-equivariant spectrum. This machine had organic input:
the $C_2$-equivariant spectrum $\mur$ of $\mathbf{R}$eal cobordism, which is found
in nature.

In the odd-primary setting, we are not so lucky. Hill-Hopkins-Ravenel, in unpublished work, 
have indicated that,
if we had a $C_3$-spectrum that displayed some of the same excellent behavior as $\mur$ does
in the $C_2$-equivariant case, we could likely resolve the 3-primary Kervaire invariant problem.
Moreover, $\mur$ and spectra built from it have had applications in homotopy theory at the
prime 2 aside
from the Kervaire invariant problem, and we could hope to achieve similar success at odd primes
if we had an analog of $\mur$.

Before tackling the odd primary story, it seems prudent to revisit the prime 2 with an eye towards
generalization. The original motivation for this paper was to answer the following:

\begin{question} Is it possible to construct
$\textup{MU}\mathbf{R}$, or the 2-local summand $\bpr$,
without using anything about manifolds or formal group laws?
\end{question}

The answer to this question non-equivariantly is yes: the first construction of $\textup{BP}$
by Brown and Peterson used only homological properties of the Steenrod algebra. Their
technique was to build $\textup{BP}$ by writing down its Adams resolution. Later, a dual construction
was found by Priddy \cite{P} that was even simpler:

\begin{thm}[Priddy, \cite{P}] The spectrum obtained from $S^0_{(p)}$ by non-trivially attaching
even, $p$-local cells in order to kill all odd homotopy groups is equivalent to $\textup{BP}$.
\end{thm}

In \S6 below, we prove the following as Theorem \ref{its-bpr}.

\begin{thm} The $C_2$-spectrum obtained from $S^0_{(2)}$ by non-trivially attaching
even dimensional, $2$-local slice 
cells in order to kill all homotopy Mackey functors in degrees of the form
$n\rho-1$ is equivalent to $\textup{BP}\mathbf{R}$. 
\end{thm}

It is important to stress that if this theorem were all we were after, the paper would be very short. Since
we know $\bpr$ exists, we can compare it to this cellular construction and see immediately that it
gives the right answer, just as in Priddy \cite[31]{P}
. But this paper is a thought experiment: can we deduce the basic properties of $\textup{BP}\mathbf{R}$
directly from its cellular construction? Let $\textup{Even}(S^0)$ denote the $2$-local,
$C_2$-equivariant spectrum obtained as in the theorem. We prove the next
result as Theorem \ref{main-bpr-thm} in the text.

\begin{thm} Without invoking the existence of $\bpr$, it is possible to establish the following
properties of $\textup{Even}(S^0)$ from its construction:
	\begin{enumerate}[(i)]
	\item The spectrum underlying $\textup{Even}(S^0)$ is $\textup{BP}$,
	\item the geometric fixed points of $\textup{Even}(S^0)$ are equivalent to $\textup{H}\mathbb{F}_2$,
	\item the homotopy Mackey functors $\rpi_{*\rho}\textup{Even}(S^0)$ are constant and given by
	$\underline{\mathbb{Z}}_{(2)}[v_1, v_2, ...]$ where $|v_i| = (2^i-1)\rho$, and $\rho$ is the regular
	representation of $C_2$.
	\item the homotopy Mackey functors $\rpi_{*\rho -1}\textup{Even}(S^0)$ vanish.
	\end{enumerate}
These last two properties determine the slice tower of $\textup{Even}(S^0)$.
\end{thm}

The main hurdle is in identifying the homology of $\textup{Even}(S^0)$. To do so,
we use obstruction theory on the slice tower to
 inductively construct power operations on $\textup{Even}(S^0)$ which constrain the homology,
and in turn the homotopy, and use this new information to show the next obstruction vanishes.
This argument is far more involved than the classical case, and requires some delicate
arguments in $RO(C_2)$-graded homological algebra.

\begin{flushleft}\textbf{Acknowledgements}. This paper benefitted greatly from conversations
with Mike Hill. I thank Doug Ravenel for sharing his notes on the odd primary Kervaire invariant
problem; this served as a guide for what sorts of results we want in the $C_p$-equivariant
case, and thus influenced the choice of methods and presentation used here. 
I am very grateful to Mike Hopkins for telling me about
Priddy's result. I would also like to thank Jim McClure for suggesting that, in order to compute the
homology of more general extended powers, I reduce to the case of a sphere instead of arguing
on the chain level. I thank Sean Tilson and Cary Malkewich for helpful comments on an earlier version
of this paper. I am also grateful to Nick Kuhn for telling me about some
of the history of the ideas I used in \S3. 
Finally, none of this would be possible without the encouragement and wisdom of my advisor,
Paul Goerss.
\end{flushleft}

\begin{flushleft}\textbf{Notation and conventions}. We display Mackey functors $M$ for $C_2$ by diagrams:
	\[
	\xymatrix{
	M(C_2/C_2)\ar@/_/[d]_{\textup{res}}\\
	M(C_2)\ar@/_/[u]_{\textup{tr}}
	}
	\]\end{flushleft}
where $M(C_2)$ is a $C_2$-module. As usual $\underline{A}$ will denote the Burnside
Mackey functor.

If $M$ is an abelian group, $\underline{M}$ will denote the associated constant Mackey functor. When
we say ``$x \in \underline{M}$'' we mean $x$ is an element of either $\underline{M}(C_2)$ or 
$\underline{M}(C_2/C_2)$.
Throughout we write $\F$ to mean $\mathbb{F}_2$.

Unless otherwise stated, the group in the background is $C_2$. For the author's sanity, we distinguish
between $C_2$ and $\Sigma_2$; hopefully this does not have the opposite effect on the reader.

\section{Preliminaries}

\subsection{$C_2$-equivariant homotopy and homology}

\subsubsection*{Fixed points} Let $\textup{E}C_2$ denote a contractible space with a free $C_2$-action, and define
$\widetilde{\textup{E}}C_2$ by the cofiber sequence of pointed $C_2$-spaces:
	\[
	\textup{E}C_{2+} \to S^0 \to \widetilde{E}C_2.
	\]
If $E$ is a $C_2$-spectrum, then we will use the following shorthand:
		\[
	E^{h}:= F(\textup{E}C_{2+}, E),\quad E_h:= \textup{E}C_{2+} \wedge E, \quad
	E^t := \widetilde{\textup{E}}C_2 \wedge  F(\textup{E}C_{2+}, E),\quad
	\Phi E := \widetilde{\textup{E}}C_2 \wedge E.
	\]
We will sometimes refer to $E^h$ as the \emph{Borel completion}
of $E$. The genuine fixed points of each spectrum above are then:
	\begin{itemize}
	\item the homotopy fixed points $E^{hC_2}$, 
	\item the homotopy orbits $E_{hC_2}$, 
	\item the Tate spectrum $E^{tC_2}$, and
	\item the geometric fixed points $\Phi^{C_2}E$.
	\end{itemize}
 This is true by definition in each case except for the equivalence:
	\[
	\left(E_h\right)^{C_2} \cong E_{hC_2}
	\]
which is a consequence of the Adams isomorphism.

\subsubsection*{Homotopy Mackey functors}

Given a representation $V$ and a $C_2$-spectrum $E$ we define a Mackey functor
by the following formula for a $C_2$-set $T$:
	\[
	\rpi_V(E) : T \mapsto [T_+ \wedge S^V, E].
	\]
Given two representations $V, W$, we define
	\[
	\rpi_{V-W}(E): T \mapsto [T_+ \wedge S^V, S^W \wedge E]. 
	\]
We will denote the value of this Mackey functor on $*$ and $C_2$ by:
	\[
	\rpi_{V-W}(E) (*) =: \rpi^{C_2}_{V-W}(E),
	\]
	\[
	\rpi_{V-W}(E)(C_2) =: \pi^u_{V-W}(E).
	\]
For every choice of isomorphism $V \oplus W' \cong V' \oplus W$ we get an isomorphism
of Mackey functors
	\[
	\rpi_{V-W}(E) \cong \rpi_{V'-W'}(E)
	\]
but this isomorphism depends on the first choice. We choose once and for all a preferred
isomorphism of a virtual representation to one of the form $a + b\sigma$, where $\sigma$
is the sign representation of $C_2$. In this way the homotopy Mackey functors of $E$
are $RO(C_2)$-graded and we denote the whole collection by
	\[
	\rpis E
	\]
When $E$ admits a pairing $E \wedge E \to E$, then $\rpis E$ admits a ring structure which
is graded commutative in the following sense: If the
degrees of $x$ and $y$ are written as $\vert x \vert = a+b\sigma$ and
$\vert y \vert = a'+b'\sigma$ then
	\[
	xy = (-1)^{aa'}\epsilon^{bb'} yx
	\]
where $\epsilon = 1-[C_2]$ in the Burnside Mackey functor $\underline{A}(C_2)$. For
a proof, see, for example, \cite[Lem. 2.12]{HK}. In particular, whenever $[C_2]-2$
acts by zero, $\epsilon$ acts by $-1$. This holds for any module spectrum over
$\hz$, for example, which covers every example of interest in this paper.

\subsubsection*{Homology of a point}

We now wish to describe the homology of a point. Before we do so we will need to name
some elements. For any representation $V$, let $a_V \in \pi^{C_2}_{-V}S^0$ denote
the element
	\[
	a_V: S^0 \to S^V
	\]
given by the inclusion of $\{0, \infty\}$. This is sometimes called an \emph{Euler class}.

For any representation $V$ of dimension $d$, the restriction map
	\[
	\hf^{C_2}_d(S^V) \to \textup{H}\F^u_d(S^V)
	\] 
is an isomorphism (\cite[Ex. 3.10]{HHR}), and we denote the unique nonzero class by $u_V$.
In particular, we have a class
	\[
	u_\sigma \in \hf^{C_2}_{1}(S^{\sigma}) = \pi^{C_2}_{1-\sigma}\hf.
	\]

We'll start by stating the calculation in the Borel complete, Tate, and geometric cases.
\begin{proposition} As $RO(C_2)$-graded Green functors, we have
	\[
	\rpis \hf^h = 
	\begin{gathered}\xymatrix{
	\mathbb{F}_2[u_{\sigma}^{\pm 1}, a_{\sigma}] \ar@/_/[d]\\
	\mathbb{F}_2[\textup{res}(u_{\sigma})^{\pm 1}] \ar@/_/[u]
	}\end{gathered}
	\quad\quad\quad
	\rpis \hf^t =
	\begin{gathered}
	\xymatrix{
	\mathbb{F}_2[u_{\sigma}^{\pm 1}, a_{\sigma}^{\pm 1}] \ar@/_/[d]\\
	0\ar@/_/[u]
	}\end{gathered}
	\]
	\[
	\rpis \Phi\hf = 
	\begin{gathered}
	\xymatrix{
	\mathbb{F}_2[u_{\sigma}, a_{\sigma}^{\pm 1}] \ar@/_/[d]\\
	0 \ar@/_/[u]
	}\end{gathered}
	\]
where $\textup{res}(a_{\sigma}) = 0$.
\end{proposition}

The case we'll be most interested in is recorded in the next proposition.

\begin{proposition} Let $\theta$ have degree $2\sigma-2$.
As an $RO(C_2)$-graded ring we have the square-zero extension:
	\[
	\pi^{C_2}_\star \hf = \F[u_{\sigma}, a_{\sigma}] \oplus 
	\F\left\{\frac{\theta}{a_\sigma^ku_{\sigma}^{n}} \vert k,n \ge 0\right\}.
	\]
The underlying homotopy groups are given by 
	\[
	\pi^u_{\star}\hf = \F[\textup{res}(u_{\sigma})^{\pm 1}]
	\]
and the Mackey functor structure is determined by $\textup{res}(a_{\sigma}) =0$ and
$\textup{tr}(\theta \textup{res}(u_{\sigma})^{-n}) = \theta u_{\sigma}^{-n}$ for $n\ge 0$. 
\end{proposition}
\begin{warning} The element $\textup{res}(u_{\sigma})^{-n}$ is not in the image of
the restriction map when $n\ge 1$. 
\end{warning}

Here is a picture of the groups $\pi^{C_2}_{a+b\sigma}\hf$:

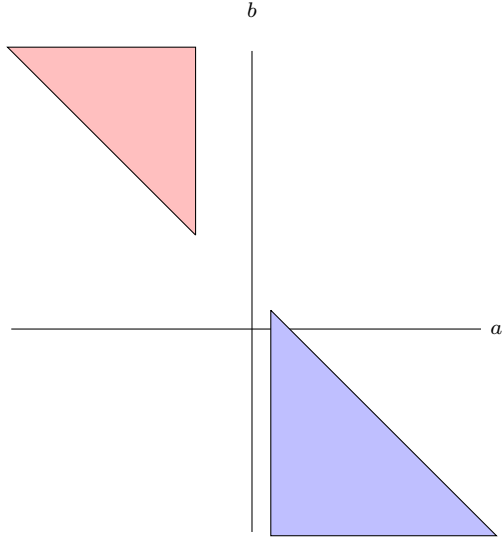
\begin{center}
\begin{tikzpicture}[scale=.5]
	\tikzstyle{every node}=[font=\footnotesize]
	\draw (-6.4,0) -- (6.4, 0);
	\draw(0, -5.4) -- (0, 7.4);
	\draw (6.5, 0) node[fill=white] {$a$};
	\draw (0, 8.5) node[fill=white] {$b$};
	\draw [fill=blue!25] (.5, .5) -- (.5, -5.5) -- (6.5, -5.5) -- (.5, .5);
	\draw [fill=red!25] (-1.5, 2.5) -- (-1.5, 7.5) -- (-6.5, 7.5) -- (-1.5, 2.5);
	\end{tikzpicture}
	\end{center}
	
The blue region is the polynomial part $\F[a_{\sigma}, u_{\sigma}]$, and the red region
is the Pontryagin dual piece. Notice the important gap $\pi^{C_2}_{-1+*\sigma}\hf = 
\pi^{C_2}_{*\rho -1}\hf = 0$. Also note that the only nonzero
group $\pi^{C_2}_{*\rho}\hf$ is $\pi^{C_2}_0\hf = \F$. 

The elements $\dfrac{\theta}{a_{\sigma}^ku_{\sigma}^{n}}$ come from the boundary map:
	\[
	\partial: \hf^t_{\star} \to \Sigma \hf_{h\star} \to \Sigma \hf_{\star}.
	\]
Specifically:
	\[
	\partial\left(\frac{1}{a_{\sigma}^{k+1}u_{\sigma}^{n+1}}\right) =
	\begin{cases}
	\frac{\theta}{a_{\sigma}^{k}u_{\sigma}^{n}} & k, n \ge 0 \\
	0 & \textup{ else}
	\end{cases}
	\]

\subsection{$C_2$-Steenrod algebra}

There are lots of versions of the dual Steenrod algebra equivariantly. We will
describe the geometric, Borel-complete, Tate, and genuine versions below. The
calculations can be found in Hu-Kriz, and the Mackey structure was computed in
\cite{R}.

\begin{convention} We let $\zeta_i \in \mathcal{A}_*$ denote Milnor's generators of the dual Steenrod algebra
(\emph{not}, as the convention has become, their conjugates!) To be absolutely clear,
the right, completed coaction on $\textup{H}^*(\mathbb{R}P^{\infty}) = \mathbb{F}_2[w]$ is given by
	\[
	w \mapsto \sum w^{2^i} \otimes \zeta_i.
	\]
\end{convention}
\begin{proposition} As a left $\Phi\hf_{\star}$-algebra,
	\[
	\rpis\Phi(\hf \wedge \hf) = \Phi \hf_{\star}[\zeta_i, \overline{u_{\sigma}}]
	\]
where the inclusion of the usual dual Steenrod algebra is a map of Hopf algebroids and
$\eta_R(a_{\sigma}) = \eta_L(a_{\sigma})$. 
\end{proposition}
\begin{proposition} As a left $\hf^h_{\star}$-algebra, we have
	\[
	\rpis (\hf \wedge \hf)^h = \hf^h_{\star}[\zeta_i]^{\hat{}}_{a_{\sigma}}.
	\]
The inclusion of the dual Steenrod algebra is a map of (completed) Hopf algebroids, and
	\[
	\eta_R(u_{\sigma}^{-1}) = \sum_{i\ge 0} u_{\sigma}^{-2^i}\zeta_ia_{\sigma}^{2^i -1},\quad
	\eta_r(a_{\sigma}) = a_{\sigma}.
	\]
The Mackey functor structure is determined by declaring that $\zeta_i$ restricts to the usual
$\zeta_i$.
\end{proposition}
\begin{corollary} As a left $\hf^t_{\star}$-algebra, we have
	\[
	\rpis (\hf \wedge \hf)^t = \hf^h_{\star}[\zeta_i]^{\hat{}}_{a_{\sigma}}[a_{\sigma}^{-1}].
	\]
The inclusion of the dual Steenrod algebra is a map of (completed) Hopf algebroids, and
	\[
	\eta_R(u_{\sigma}^{-1}) = \sum_{i\ge 0} u_{\sigma}^{-2^i}\zeta_ia_{\sigma}^{2^i -1},\\
	\eta_r(a_{\sigma}) = a_{\sigma}.
	\]
\end{corollary}

In this paper we will mostly be concerned with the Hopf algebroid
$\rpis(\hf \wedge \hf)$, which we now name.

\begin{definition} The {\bf $C_2$-equivariant dual Steenrod algebra}
is the Hopf algebroid in Green functors
$(\hf_{\star}, \rpis(\hf \wedge \hf))$. We will
denote it by $\arep$.
\end{definition}

We will need to define a few elements in $\arep$ before stating the computation, and
for that we need the equivariant analog of $\textup{B}C_2$.

\begin{definition} Let $\textnormal{E}\mu_2$ be the $C_2\times \Sigma_2$-homotopy type uniquely
determined by the property that
	\[
	\left(\textnormal{E}\mu_2\right)^H = \begin{cases}
	* & H = 1, C_2, \Delta\\
	\varnothing & H = \Sigma_2, C_2 \times \Sigma_2
	\end{cases}
	\]
where $\Delta \subset C_2 \times \Sigma_2$ is the graph of the unique nontrivial homomorphism
from $C_2$ to $\Sigma_2$. Let $\textup{B}\mu_2$ be the $C_2$-homotopy type $\textnormal{E}\mu_2/\Sigma_2$.
\end{definition}
\begin{remark} The space underlying $\textup{B}\mu_2$ is $\mathbb{R}P^{\infty}$, but the $C_2$ action is
nontrivial. Indeed, $(\textup{B}\mu_2)^{C_2} \cong \mathbb{R}P^{\infty} \amalg \mathbb{R}P^{\infty}$. This space
has many other names. In \cite{HK} it is called $\textup{B}'\mathbb{Z}/2$, and it is also the $C_2$-equivariant
classifying space for $C_2$-equivariant $\Sigma_2$-bundles, which is often written $\textup{B}_{C_2}\Sigma_2$.
\end{remark}

Before describing the dual Steenrod algebra, we'll need a few preliminaries on the cohomology of $\textup{B}\mu_2$.
Consider $\mathbb{C}P^{\infty}$ as a $C_2$-space under complex conjugation. Then the Euler class
of the canonical bundle lifts canonically to an element in $\hz^{\rho}\mathbb{C}P^{\infty}$.
The natural inclusion $\textup{B}\mu_2 \hookrightarrow \mathbb{C}P^{\infty}$ is equivariant and so defines an
element $\widetilde{b} \in \hz^{\rho}\textup{B}\mu_2$ by pulling back this Euler class. Changing
coefficients yields an element $b \in \hf^{\rho}\textup{B}\mu_2$.

Let
$\beta$ denote the Bockstein associated to the exact sequence of Mackey functors
	\[
	\underline{\mathbb{Z}} \stackrel{\cdot 2}{\longrightarrow} \underline{\mathbb{Z}}
	\longrightarrow \underline{\F}.
	\]
Then there is a unique element $c \in \hf^{\sigma}\textup{B}\mu_2$ such that $\beta c = b$ and $c$ vanishes
when restricting to a point. The existence of some such $c$ follows from the fact that $2\widetilde{b}=0$,
which one proves by trivializing the square of the line bundle classified by $\widetilde{b}$. There are exactly
two choices of $c$, differing by addition of $a_{\sigma}$, and this is detected by restricting to a point.

Hu-Kriz show that $\hf^{\star}\textup{B}\mu_2$ is free
over $\hf_{\star}$ with basis elements $c^{\epsilon}b^i$ for $i\ge 0$ and 
$\epsilon = 0,1$.\footnote{In \cite{HK}, they denote by $b'$ what we have denoted by $b$, and
they never quite pin down which choice of $c$ they use.}
We give a different
proof of this fact below, as well (see Proposition \ref{coh-bmu}, below).

The key observation is that the right, completed Steenrod coaction on $c$ takes the form
	\[
	\psi(c) = c \otimes 1 + \sum_{i\ge 0} b^{2^i} \otimes \tau_i,
	\]
for some elements $\tau_i \in \arep$ of degree $(2^i -1)\rho + 1$. The next theorem
is the content of \cite[Thm. 6.41]{HK}.

\begin{theorem}[Hu-Kriz] \label{arep-structure}
Let $\xi_i :=\beta \tau_i$. Then, as a left $\hf_{\star}$-algebra,
\[
	\arep = \hf_{\star}[\tau_i, \xi_i, \overline{u_{\sigma}}]/
	\tau_0a_{\sigma} = \overline{u_{\sigma}} + u_{\sigma}, \, 
	\tau_i^2 = \tau_{i+1}a_{\sigma} + \xi_{i+1}\overline{u_{\sigma}}.
	\]
Here $\vert \tau_i\vert = (2^i - 1)\rho +1$, and $\vert \xi_i\vert = (2^i - 1)\rho$. The behavior
of the right unit can be written in terms of the boundary map
$\partial: \hf^t_{\star} \to \Sigma \hf_{h\star} \to \Sigma \hf_{\star}$:
	\[
	\eta_R\left(\frac{\theta}{a_{\sigma}^ku_{\sigma}^{n}}\right)
	=
	\partial\left(\frac{1}{a_{\sigma}^{k+1}(u_{\sigma} + \tau_0a_{\sigma})^{n+1}}\right)
	= \frac{\theta}{a_{\sigma}^ku_{\sigma}^n} + 
	\frac{(n+1)a_{\sigma}\theta}{a_{\sigma}^{k}u_{\sigma}^{n+1}}\tau_0 
	+ \cdots
	\]
when $k,n\ge 0$.
\end{theorem}

\begin{remark} For us, the most important element is $\tau_0$, and that has a more elementary description.
The left and right units on $u_{\sigma}$ restrict to the same element on underlying homotopy,
so $u_{\sigma} + \overline{u_{\sigma}}$ is in the kernel of the restriction. Thus there is some
element $\tau_0 = \frac{1}{a_{\sigma}}(u_{\sigma} + \overline{u_{\sigma}})$. Since the
transfer is zero in this degree, the element $\tau_0$ is uniquely determined by this property.

 Actually, this basic idea lies behind the whole computation of the Dyer-Lashof and dual Steenrod
algebras: write down some classical elements and some equivariant elements that restrict to
the same thing and keep dividing by $a_{\sigma}$ until you can't any more. This business of
dividing by $a_{\sigma}$ corresponds, on the cellular side, to extending a map over a $CW$-complex
with a free cell in each dimension. Such things end up looking like representation spheres in the
case of $C_2$ and for $C_p$ they're stranger objects. These cell complexes appear in the
Hill, Hopkins, and Ravenel computation \cite{HHR2} of $\pi_*EO_{2(p-1)}$  and in 
unpublished work by the same authors.
\end{remark}

\section{Extended Powers}

In this section we refine the skeletal filtration of $\textup{E}\Sigma_2$ to an equivariant
filtration of $\textup{E}\mu_2$. It's important that we do not use the standard
equivariant skeletal filtration on $\textup{E}\mu_2$. To see why, notice that the $0$-skeleton
of $\textup{E}\mu_2$ must contain at least two 0-cells:
	\[
	\Sigma_{2+}, \quad \frac{C_2 \times \Sigma_2}{\Delta}.
	\]
In terms of power operations, the first of these cells is eventually responsible for the \emph{squaring}
map, while the second is responsible for the \emph{norm} map. We choose a filtration which
priviledges the squaring map in order to get formulas which look like the classical case.

Another departure from the classical case is that the resulting filtration on $\hf \wedge \mathbb{P}_2(X)$
does not obviously split in general. However, the filtration does split when $X = S^{n\rho}$ or
$S^{n\rho-1}$, and this ends up being all we need in the sequel.

\subsection{Definitions and first properties}
Let $\mathcal{O}$ denote an operad in $C_2$-spaces such that $\mathcal{O}(n)$ is a universal
space for the family of subgroups $\Gamma \subset C_2 \times \Sigma_n$ which are graphs
of homomorphisms $C_2 \longrightarrow \Sigma_n$. This is a
complete $N_{\infty}$-operad for the group $C_2$ in the terminology of \cite{BH}, or,
a $G\text{-}E_{\infty}$-operad in more classical terminology. There is a unique such operad
up to equivalence so we are not concerned with the choice. Notice that, in particular, we have
an equivalence of $C_2 \times \Sigma_2$-spaces
	\[
	\mathcal{O}(2) \cong \textnormal{E}\mu_2.
	\] 
Let $R$ be an equivariant commutative ring spectrum
, and let $\textup{Mod}_R$ denote the category of modules over (a cofibrant replacement of) $R$.
We have a functor
	\[
	\mathbb{P}: \textup{Mod}_R \longrightarrow \textup{Alg}_{\mathcal{O}}(\textup{Mod}_R)
	\]
homotopically left adjoint to the forgetful functor. 

The underlying object in $\textup{Mod}_R$, $\mathbb{P}(X)$, admits a canonical filtration by arity 
with subquotients $\mathbb{P}_m(X)$ for $m\ge 0$ and we have
	\[
	\mathbb{P}_2(X) = \mathcal{O}(2)_+ \otimes_{\Sigma_2} X^{\otimes 2} \cong \textnormal{E}\mu_2 \otimes_{\Sigma_2}
	X^{\otimes 2}.
	\]
We will refer to this as the \textbf{genuine} or \textbf{complete extended power} of $X$. Notice that it differs
from the standard extended power $X^{\otimes 2}_{h\Sigma_2}$. If we wish to emphasize
the ring we'll use the notation $\mathbb{P}_2^{R}$.

\begin{warning} While $\Phi^{H}$ does take
$N_{\infty}$-algebras to $N_{\infty}$-algebras (for a different operad), it \emph{does not
preserve extended powers} unless $H=e$ is trivial. That is:
	\[
	\Phi^{H}(\mathbb{P}_m^{\mathcal{O}}(X))  \not\cong \mathbb{P}_m^{\mathcal{O}^{H}}(\Phi^{H}X),
	\]
where we've momentarily decorated the extended powers to indicate which operad we're using.
For example, when $X = S^0$, $G=H=C_2$, and $m=2$, the left hand side is (the suspension
spectrum of) $\mathbb{R}P^{\infty} \amalg
\mathbb{R}P^{\infty}$ and the right hand side is just (the suspension spectrum of) $\mathbb{R}P^{\infty}$.
\end{warning}

Since change of base ring is monoidal, we get the following. 

\begin{lemma}\label{lem-change-base} If $R \to R'$ is a map of commutative ring spectra, then there is an equivalence in
the homotopy category:
	\[
	R' \wedge_{R} \mathbb{P}_2^{R}(X) \cong \mathbb{P}_2^{R'}((R' \wedge_{R} X)).
	\]
\end{lemma}

We also record the behavior of complete extended powers with respect to indexed wedges, for later use.
To prove it one may, for example, model the extended power by honest symmetric powers
after cofibrant replacement \cite[B.117]{HHR} and then argue using the distributive law \cite[A.37]{HHR}.
\begin{lemma}\label{lem-distr} For any equivariant commutative ring $R$, 
and $R$-modules $X$ and $Y$ we have
	\[
	\mathbb{P}_2(X \vee Y) \cong \mathbb{P}_2(X) \vee \mathbb{P}_2(Y) \vee (X \vee Y),
	\]
	\[
	\mathbb{P}_2(C_{2+} \wedge X) \cong C_{2+} \wedge \mathbb{P}_2(X) \vee (C_{2+} \wedge X).
	\]
\end{lemma}
\subsection{A filtration on $\textnormal{E}\mu_2$}
We will be interested in filtering $\mathbb{P}_2(X)$ when $X$ is a representation sphere in order to get
a handle on its homology. These filtrations will be induced by a filtration on $\textnormal{E}\mu_2$ itself
which refines the non-equivariant skeletal filtration. We give two
descriptions of the filtration- one geometric, and the other combinatorial. 

\begin{construction}\label{geometric-filtration}
Let $\tau$ denote the sign representation of $\Sigma_2$, and define
	\[
	F_{2k}\textnormal{E}\mu_2 := S((k\rho + 1) \otimes \tau), \quad F_{2k+1}\textnormal{E}\mu_2 := S((k+1)\rho \otimes \tau).
	\]
There are evident equivariant inclusions $F_{i}\textnormal{E}\mu_2 \subset F_{i+1}\textnormal{E}\mu_2$ which yield a 
filtered $C_2\times \Sigma_2$-space $F_{\infty}\textnormal{E}\mu_2$. We write $\textup{gr}_i\textnormal{E}\mu_2$
for the $i$th-layer of the associated graded. Recall that $\Delta \subset C_{2} \times \Sigma_2$
denotes the diagonal subgroup.

We'll now build equivariant maps
	\[
	g_{2k}: S^{k\rho} \wedge \Sigma_{2+} \longrightarrow \textup{gr}_{2k}\textnormal{E}\mu_2,
	\]
	\[
	g_{2k+1}: S^{(k\rho + 1)\tau} \wedge \left(\frac{C_2 \times \Sigma_2}{\Delta}\right)_+
	\longrightarrow \textup{gr}_{2k+1}\textnormal{E}\mu_2.
	\]
For $g_{2k}$, let $D(k\rho) \longrightarrow S(k\rho +1)$ be the $C_2$-equivariant inclusion of the graph of
the function $D(k\rho) \longrightarrow \mathbb{R}$ given by $v \mapsto \sqrt{1 - |v|^2}$. 
By adjunction, this extends
to a map of pairs
	\[
	(D(k\rho) \times \Sigma_2, S(k\rho) \times \Sigma_2) \longrightarrow 
	(S((k\rho +1)\otimes \tau), S(k\rho \otimes \tau)) = 
	(F_{2k}\textnormal{E}\mu_2, F_{2k-1}\textnormal{E}\mu_2)
	\]
and so descends to define $g_{2k}$. 

We define $g_{2k+1}$ similarly. That is, we use the functions $v \mapsto \pm\sqrt{1-|v|^2}$ to define
a map 
	\[
	D((k\rho+1) \otimes \tau) \times \left(\frac{C_2 \times \Sigma_2}{\Delta}\right) \longrightarrow
S((k+1)\rho \otimes \tau),\]
which includes into the top hemisphere on the trivial coset and the bottom
hemisphere on the nontrivial coset. This map descends to the quotient and yields $g_{2k+1}$.

\end{construction}
\begin{proposition}\label{proposition-filtration}
 The $C_2 \times \Sigma_2$ space $F_{\infty}\textnormal{E}\mu_2 = S(\infty (\rho \otimes \tau))$ is
a model for $\textnormal{E}\mu_2$, and the maps
	\[
	g_{2k}: S^{k\rho} \wedge \Sigma_{2+} \longrightarrow \textup{gr}_{2k}\textnormal{E}\mu_2,
	\]
	\[
	g_{2k+1}: S^{(k\rho + 1)\tau} \wedge \left(\frac{C_2 \times \Sigma_2}{\Delta}\right)_+
	\longrightarrow \textup{gr}_{2k+1}\textnormal{E}\mu_2.
	\]
are equivalences of pointed $(C_2 \times \Sigma_2)$-spaces.
\end{proposition}
\begin{proof} For the first claim note that the underlying space is the sphere in $\mathbb{R}^{\infty}$, the $\Delta$-fixed points are given by $S(\infty (\sigma \otimes \tau))$, and the $C_2$-fixed points are given by $S(\infty(1 \otimes \tau))$, all of which are contractible. It is also evident that $\Sigma_2$ acts freely. 

For the latter claim, it's clear that $g_{2k}$ and $g_{2k+1}$ are homeomorphisms, so we'll just double-check
that they are equivariant. Let $\gamma$
be the generator of $C_2$ and $\chi$ the generator of $\Sigma_2$. We will abuse notation and evaluate
$g_{2k}$ (resp. $g_{2k+1}$) on points of $D(k\rho) \times \Sigma_2$ (resp. $D((k\rho +1) \otimes \tau)
\times \left(\dfrac{C_2 \times \Sigma_2}{\Delta}\right)$). 
	\begin{align*}
	g_{2k}(\gamma\cdot(v_1, v_2, ..., v_{2k}, 1)) &= g_{2k}(v_1, -v_2, ..., v_{2k-1}, -v_{2k}, 1)\\
	&= (v_1, -v_2, ..., v_{2k-1}, -v_{2k}, \sqrt{1-|v|^2})\\
	&= \gamma \cdot (v_1, v_2, ..., v_{2k}, \sqrt{1-|v|^2})\\
	&= \gamma \cdot g_{2k}(v_1, ..., v_{2k}, 1)
	\end{align*}
	\begin{align*}
	g_{2k}(\chi \cdot (v_1, ..., v_{2k}, 1)) &= g_{2k}(v_1, v_2, ..., v_{2k}, \chi)\\
	&= (-v_1, ..., -v_{2k}, -\sqrt{1-|v|^2})\\
	&= \chi \cdot (v_1, ..., v_{2k}, \sqrt{1-|v|^2}) \\
	&= \chi \cdot g_{2k}(v_1, ..., v_{2k}, 1)
	\end{align*}
	\begin{align*}
	g_{2k+1}(\gamma \cdot (v_1, ..., v_{2k}, v_{2k+1}, [1])) &= 
	g_{2k+1}(v_1, -v_2, ..., -v_{2k}, v_{2k+1}, [\gamma])\\
	&= (v_1, -v_2, ..., v_{2k+1}, -\sqrt{1-|v|^2})\\
	&= \gamma \cdot (v_1, ..., v_{2k+1}, \sqrt{1-|v|^2})\\
	&= \gamma \cdot g_{2k+1}(v_1, ..., v_{2k+1}, [1])
	\end{align*}
	\begin{align*}
	g_{2k+1}(\chi \cdot (v_1, ..., v_{2k+1}, [1])) &=
	g_{2k+1}(-v_1, ..., -v_{2k+1}, [\chi])\\
	&= (-v_1, ...., -v_{2k+1}, -\sqrt{1-|v|^2})\\
	&= \chi \cdot (v_1, ..., v_{2k+1}, \sqrt{1-|v|^2})\\
	&= \chi \cdot g_{2k+1}(v_1, ..., v_{2k+1}, [1])
	\end{align*}
\end{proof}
\begin{remark}\label{remark:untwist}
It's worth pointing out that if $X$ is a pointed $G$-space and $H \subset G$ is a subgroup
then $X \wedge \left(G/H\right)_+$ is equivalent to $i_H^*X \wedge_H G_+$ as $G$-spaces. 
Indeed, the adjunction between restriction and induction gives a $G$-equivariant map from the right-hand
side to the left which is evidently a homeomorphism.

Now, $\sigma$ and $\tau$ both restrict to the sign representation of $\Delta$, and
$\sigma\tau$ restricts to the trivial representation of $\Delta$. Combining these two
observations yields an equivalence:
	\[
	S^{\tau V} \wedge \left(\dfrac{C_2 \times \Sigma_2}{\Delta}\right)_+ \cong S^{\sigma V}
	\wedge \left(\dfrac{C_2 \times \Sigma_2}{\Delta}\right)_+.
	\]
\end{remark}
Now we sketch a combinatorial approach to the above filtration. This has the benefit of generalizing
more readily to the odd primary case, but other than that it is much the same.

\begin{construction} Let $\mathbf{E}\dfrac{C_2 \times \Sigma_2}{\Delta}$ denote the
$(C_2 \times \Sigma_2)$-simplicial set with $n$-simplices given by
	\[
	\left(\mathbf{E}\frac{C_2 \times \Sigma_2}{\Delta}\right)_n = \left(\frac{C_2\times \Sigma_2}{\Delta}
	\right)^{\times (n+1)}
	\]
and boundaries and degeneracies given by projection and diagonal maps, respectively.

Similarly let $\mathbf{E}\Sigma_2$ denote the nerve of the transport category for $\Sigma_2$, i.e.
the $(C_2 \times \Sigma_2)$-simplicial set with $n$-simplices given by
	\[
	\left(\mathbf{E}\Sigma_2\right)_n = \Sigma_2^{\times (n+1)}
	\]
and boundaries and degeneracies given by projections and diagonals.

\begin{remark} This is slightly non-standard: we are thinking of the tuple $(g_0, ..., g_n)$ as
specifying the string of morphisms $g_0 \to g_1 \to \cdots \to g_n$ and the $\Sigma_2$-action
is diagonal. We could equally well use the tuples $(g_0, h_1, ..., h_n)$ where $g_0$ is
the source of the first arrow, and we have $g_0h_1 = g_1$, $g_1h_2 = g_2$, etc. Then $\Sigma_2$
acts only on the first coordinate.
\end{remark}

Finally, let $\mathbf{E}\mu_2$ denote the join of these two simplicial sets. So the $n$-simplices are:
	\[
	\left(\mathbf{E}\mu_2\right)_n =  \Sigma_2^{\times (n+1)}
	\cup \, \left(\frac{C_2\times \Sigma_2}{\Delta}
	\right)^{\times (n+1)}
	 \, \cup \bigcup_{i+j = n-1} 
	 	 \Sigma_2^{\times (i+1)}
	 \times \, \left(\frac{C_2\times \Sigma_2}{\Delta}\right)^{\times (j+1)}.
	\]
\end{construction}

\begin{lemma} The geometric realization of $\mathbf{E}\mu_2$ is a model for $\emu$.
\end{lemma}
\begin{proof} This is immediate from the fact that taking fixed points commutes with
the join and that joining with a contractible space yields a contractible space.
\end{proof}
\begin{lemma} The non-degenerate simplices of $\mathbf{E}\mu_2$ are given by
\begin{align*}
	N\left(\mathbf{E}\mu_2\right)_n &= \Sigma_2
	\cup  \left(\frac{C_2 \times \Sigma_2}{\Delta}\right)\cup 
	\bigcup_{i+j = n-1}\Sigma_2 
	\times \left(\frac{C_2 \times \Sigma_2}{\Delta}\right)\\
	&= \Sigma_2
	\cup  \left(\frac{C_2 \times \Sigma_2}{\Delta}\right)\cup 
	\bigcup_{i+j = n-1} C_2 \times \Sigma_2.
	\end{align*}
\end{lemma}
\begin{proof} This follows from the computation of the non-degenerate simplices in each
factor of the join, which is classical:
\[
	N(\mathbf{E}\frac{C_2 \times \Sigma_2}{\Delta})_n = \left(\frac{C_2 \times \Sigma_2}{\Delta}\right),
	\]
	\[
	N(\mathbf{E}\Sigma_2)_n = \Sigma_2.
	\]
\end{proof}

To give the filtration in this combinatorial setting we need some notation for
simplices in $\mathbf{E}\mu_2$. Let $x_n$ denote a
generator of the
non-degenerate orbit $\Sigma_2$ in dimension $n$, and let
$y_n$ denote a generator of the non-degenerate orbit $\dfrac{C_2 \times \Sigma_2}{\Delta}$ in dimension $n$. Denote
by $x_i \ast y_j$ the corresponding generators of $C_2 \times \Sigma_2$ in dimension
$i+j+1$. 

Now define $F_{2k}\mathbf{E}\mu_2$ as the subcomplex spanned by the simplices
$x_0, ..., x_k, y_0, ..., y_{k-1}$ and their pairwise joins, and define $F_{2k+1}\mathbf{E}\mu_2$ as the
subcomplex spanned by the simplices $x_0, ..., x_k, y_0, ..., y_k$ and their pairwise joins.

\begin{construction} Begin with the pointed simplicial $C_2$-set $(\textup{sk}_n\mathbf{E}C_2)_+$,
take the Kan suspension, induce up to a pointed simplicial $(C_2 \times \Sigma_2)$-set by smashing
with $\Sigma_{2+}$, and then smash with $\Delta^k/\partial\Delta^k$. The result is a simplicial
$(C_2 \times \Sigma_2)$-set $\mathbf{G}_{2k}$ with the following properties: 
	\begin{enumerate}[(i)]
	\item the geometric
realization of $\mathbf{G}_{2k}$ is equivalent to $S^k \wedge S^{k\sigma} \wedge \Sigma_{2+}$,
	\item the non-degenerate simplices are: the basepoint, a copy of
	$\Sigma_2$ in dimension $k$ generated by a simplex $a_k$, a copy of $C_2 \times \Sigma_2$
	in each dimension $k+1, ..., k+2$ generated by simplices $z_{0}, ..., z_{k-1}$, respectively.
	\end{enumerate}
One can construct an equivariant map
	\[
	\mathbf{G}_{2k} \longrightarrow \textup{gr}_{2k}\mathbf{E}\mu_2
	\]
which sends $a_k \mapsto x_k$ and $z_i \mapsto x_k \ast y_i$ and induces an equivariant map
	\[
	h_{2k}: S^k \wedge S^{k\sigma} \wedge \Sigma_{2+}
	\longrightarrow \textup{gr}_{2k}\emu
	\]
upon taking geometric realization.

The same construction with $\mathbf{E}\Sigma_2$ replacing $\mathbf{E}C_2$ and 
$\left(\dfrac{C_2 \times \Sigma_2}{\Delta}\right)_+$
replacing $\Sigma_{2+}$ gives a map
	\[
	h_{2k+1}: S^k \wedge S^{(k+1)\tau} \wedge \left(\dfrac{C_2 \times \Sigma_2}{\Delta}\right)_+
	\longrightarrow \textup{gr}_{2k+1}\emu.
	\]
\end{construction}
\begin{proposition} The maps $h_{i}$ are equivalences of pointed $(C_2 \times \Sigma_2)$-spaces.
\end{proposition}
\begin{proof} By construction the maps induce isomorphisms on non-degenerate simplices, which
implies the result.
\end{proof}
\begin{remark} To connect the combinatorial and geometric constructions notice that the subcomplex
spanned by $x_0, ..., x_k$ corresponds to a cell structure on $S(k(1 \otimes \tau))$ and the subcomplex
spanned by $y_0, ..., y_k$ corresponds to a cell structure on $S(k(\sigma \otimes \tau))$. The join of
spheres has the following property: if $V$ and $W$ are vector spaces then there is a natural homeomorphism
	\[
	S(V) \ast S(W) \cong S(V \oplus W).
	\]
From this it should be clear that the two filtrations and the identifications of their associated graded spaces
are compatible up to re-ordering coordinates and applying Remark \ref{remark:untwist}.
\end{remark}
\subsection{Extended powers of spheres}
The filtration on $\emu$ introduced in the previous section gives rise to a filtration $F_k\mathbb{P}_2X$
on complete extended powers which is natural in $X$. We'll be particularly interested in the case when
$X$ is a (virtual) representation sphere. We begin by determining the first differential in the spectral
 sequence associated to this filtration on $\mathbb{P}_2X$ in general.

\begin{theorem} The filtration $\{F_k\mathbb{P}_2\}$ of the functor
$\mathbb{P}_2$ has the following properties:
	\begin{enumerate}[(i)]
	\item The maps $g_k$ of (\ref{geometric-filtration})
	yield equivalences:
		\[
		\mathrm{gr}_{2k}\mathbb{P}_2 \cong S^{k\rho} \wedge (-)^{\wedge 2},
		\]
		\[
		\mathrm{gr}_{2k+1}\mathbb{P}_2 \cong S^{k\rho +\sigma} \wedge N^{C_2}(-).
		\]
	\item For $X \in \mathrm{Sp}^{C_2}$, the connecting map 
		\[
	S^{k\rho} \wedge X^{\wedge 2}
	\cong
	\mathrm{gr}_{2k}\mathbb{P}_2(X) \longrightarrow 
	\Sigma\mathrm{gr}_{2k-1}\mathbb{P}_2(X)
	\cong S^{k\rho} \wedge N^{C_2}X
		\]
	is obtained from the map
		\[
		\xymatrix{
		\downarrow_1 X^{\wedge 2} \ar[r]^{\mathrm{id} \wedge \gamma}
		& \downarrow_1 N^{C_2}X 
		}
		\]
	by taking the transfer and tensoring with $S^{k\rho}$. That is, the connecting
	map is the composite:
		\[
		\xymatrix{
		S^{k\rho} \wedge X^{\wedge 2}
		\ar[r]& S^{k\rho} \wedge \uparrow^{C_2} \downarrow_1 X^{\wedge 2}
		\ar[rr]^{\mathrm{id} \wedge \uparrow^{C_2} (\mathrm{id} \wedge \gamma)} && 
		 S^{k\rho} \wedge \uparrow^{C_2} \downarrow_1 N^{C_2}X \ar[r] & 
		 S^{k\rho} \wedge N^{C_2}X
		 }
		\]
	\item For $X \in \mathrm{Sp}^{C_2}$, the connecting map
		\[
		S^{k\rho +\sigma} \wedge N^{C_2}X \cong
		\mathrm{gr}_{2k+1}\mathbb{P}_2(X)
		\longrightarrow \Sigma \mathrm{gr}_{2k+2}\mathbb{P}_2(X)
		\cong S^{k\rho + 1} \wedge X^{\wedge 2}
		\]
	is obtained from the map
		\[
		\xymatrix{
		\downarrow_1 N^{C_2}X \ar[r]^{\mathrm{id} \wedge \gamma}
		& \downarrow_1 X^{\wedge 2} 
		}
		\]
	by taking the $\sigma$-twisted transfer 
	 and smashing with $S^{k\rho}$. That is,
	the connecting map is the composite:
		\[
		\xymatrix{
		S^{k\rho + \sigma} \wedge N^{C_2}(X)\ar[r] &
		S^{k\rho} \wedge \uparrow^{C_2}\downarrow_1\left(S^{\sigma} \wedge
		N^{C_2}X\right) 
		\ar[rrr]^-{\mathrm{id} \wedge \uparrow^{C_2}(\mathrm{id}\wedge \mathrm{id}
		\wedge \mathrm{\gamma})}&&&
		S^{k\rho} \wedge \uparrow^{C_2}\downarrow_1\left( S^1 \wedge
		X^{\wedge 2}\right) \ar[r] &
		S^{k\rho +1} \wedge X^{\wedge 2}
		} 
		\]
	\end{enumerate}
\end{theorem}
\begin{proof} We will reduce to a non-equivariant calculation.
First note the following two sequences of isomorphisms, where
we decorate hom-sets to indicate which category we are in:
	\begin{align*}
	\left[\uparrow_{C_2}^{C_2 \times \Sigma_2}S^{k\rho}, 
	\uparrow^{C_2\times\Sigma_2}_{\Delta} 
	S^{((k-1)\rho+1)\tau+1}\right]^{C_2\times \Sigma_2}
	&=
	\left[S^{k\rho}, 
	\downarrow^{C_2\times\Sigma_2}_{C_2}\uparrow^{C_2\times\Sigma_2}_{\Delta} 
	S^{((k-1)\rho+1)\tau+1}\right]^{C_2}\\
	&\cong 
	\left[S^{k\rho}, 
	\uparrow^{C_2}\downarrow_1^{C_2} 
	S^{k\rho}\right]^{C_2}\\
	&=\left[S^{2k}, 
	S^{2k}\right]
	\end{align*}
	\begin{align*}
	\left[S^{(k\rho+1)\tau} \wedge \left(\frac{C_2\times\Sigma_2}{\Delta}\right)_+,
	\uparrow_{C_2}^{C_2\times\Sigma_2}S^{k\rho +1}\right]^{C_2\times\Sigma_2}
	&=
	\left[\downarrow^{C_2\times\Sigma_2}_{C_2}\left(S^{(k\rho+1)\tau} 
	\wedge \left(\frac{C_2\times\Sigma_2}{\Delta}\right)_+\right),
	S^{k\rho +1}\right]^{C_2}\\
	&\cong
	\left[\uparrow^{C_2}_1\downarrow^{C_2}_1S^{k\rho+\sigma},
	S^{k\rho +1}\right]^{C_2}\\
	&=
	\left[S^{2k+1},
	S^{2k+1}\right].
	\end{align*}
The upshot is that the connecting maps
	\[
	S^{k\rho} \wedge \Sigma_{2+} \to S^{((k-1)\rho+1)\tau +1} \wedge
	\left(\frac{C_2\times\Sigma_2}{\Delta}\right)_+
	\]
	\[
	S^{(k\rho+1)\tau} \wedge \left(\frac{C_2\times\Sigma_2}{\Delta}\right)_+
	\to S^{k\rho+1} \wedge \Sigma_{2+}
	\]
are determined by their underlying behavior on trivial cosets. Our filtration on $\emu$ refines the classical filtration
on $E\Sigma_2$. There, we know that the connecting maps
are given by $1+\chi$ and $1-\chi$, which both have trivial component
of degree 1. The result now follows from tracing back through the isomorphisms.
\end{proof}

In the case of extended powers of spheres, we can say much more.

\begin{theorem}\label{thm-extended-power-filtration}
 The associated graded space of $\mathbb{P}_2(S^V)$ with respect to the above
filtration is given by
	\[
	\textup{gr}_{2k}\mathbb{P}_2(S^V) \cong S^{k\rho} \wedge S^{2V},
	\]
	\[
	\textup{gr}_{2k+1}\mathbb{P}_2(S^V) \cong S^{k\rho + \sigma} \wedge S^{\rho V}.
	\]
If, moreover, $V = m\rho$ or $V=m\rho -1$, this filtration splits upon smashing with $\hf$ and we get
an equivalence of $C_2$-spectra:
	\[
	\hf \wedge \mathbb{P}_2(S^{m\rho}) \cong
	S^{2m\rho} \wedge \left(\bigvee_{k\ge 0} \hf \wedge S^{k\rho} \vee 
	\bigvee_{k\ge 0} \hf \wedge S^{k\rho +\sigma}\right),
	\]
	\[
	\hf \wedge \mathbb{P}_2(S^{m\rho-1}) \cong
	S^{2m\rho-2} \wedge \left(\bigvee_{k\ge 0} \hf \wedge S^{k\rho} \vee 
	\bigvee_{k\ge 0} \hf \wedge S^{k\rho+1}\right).	
	\]
\end{theorem}

Notice that this situation differs from its classical counterpart. We only have a `Thom isomorphism'
for regular representation spheres. This is analogous to the classical odd primary situation where
only the extended powers of even spheres have canonical orientations.

\begin{remark} It is nevertheless true that the homology of $\mathbb{P}_2(S^V)$
is free over the homology of a point when $V$ is a representation. This follows from
the surprisingly general
freeness theorem of Kronholm \cite{Kr}. However, it is not clear a priori
in which degrees we may choose our generators and, in any case, we do
not need this fact for any of our arguments. 
We are grateful to Sean Tilson
for telling us about this reference, which we did not know of while writing the
first version of this paper.
\end{remark}

Before giving the proof, let's record the following consequence of the computation of $\hf_{\star}$.

\begin{lemma}\label{lem-c2-vanish-range} We have vanishing ranges:
	\[
	\underline{[S^{t\rho +\sigma}, \hf]} = 0, \quad t \in \mathbb{Z},
	\]
	\[
	\underline{[S^{t\rho -1}, \hf]} = 0, \quad t \in \mathbb{Z},
	\]
	\[
	\underline{[S^{t\rho} \wedge S^{(1-\sigma)(a+b\sigma)}, \hf]} = 0, \quad t\ge 1,\quad a-b\ge -2 
	\]
\end{lemma}
\begin{proof} The first two are equivalent and straightforward. For the last one, notice that,
since $\sigma^2 = 1$, we get $(1-\sigma)(a+b\sigma) = (a-b)(1-\sigma)$. The coefficients $\hf_{\star}$ don't necessarily vanish on this line, but they do as soon
as you add a positive multiple of the regular representation when $a-b\ge -2$.
\end{proof}

\begin{proof}[Proof of Theorem \ref{thm-extended-power-filtration}] The identification
of the associated graded is a special case of the previous theorem.

Let $V = a+b\sigma$ with $a-b = 0$ or $-1$. We need to show that,
for all $n\ge 0$, the boundary map
	\[
	\hf \wedge \mathrm{gr}_{n}\mathbb{P}_2\left(S^V\right)
	\to \hf \wedge \Sigma F_{n-1}\mathbb{P}_2\left(S^V\right)
	\]
is null. It suffices to show that 
	\begin{enumerate}[(i)]
	\item The connecting map
		\[
		\mathrm{gr}_{n}\mathbb{P}_2\left(S^V\right)
	\to \hf \wedge \Sigma \mathrm{gr}_{n-1}\mathbb{P}_2\left(S^V\right)
		\]
	is null.
	\item The groups
		\[
		\left[\mathrm{gr}_{n}\mathbb{P}_2\left(S^V\right),
		\hf \wedge \Sigma \mathrm{gr}_{m}\mathbb{P}_2\left(S^V\right)\right]
		\]
	vanish for $m\le n-2$. 
	\end{enumerate}
Suppose that $n=2k$ is even. Then the map in (i) takes the form
	\[
	S^{k\rho+2V} \longrightarrow \hf \wedge S^{(k+|V|)\rho},
	\]
which corresponds to an element in $\pi^{C_2}_{2V - |V|\rho}\hf$. We have assumed
that $V = |V|\rho$ or $V = (|V|+1)\rho -1$. In either case,
the Mackey functor $\underline{\pi}_{2V-|V|\rho}\hf$ has injective restriction map,
so we need only show that the \emph{underlying} homotopy class vanishes. 
But this follows from the classical calculation, since our filtration refines the
classical filtration of $B\Sigma_2$. The same 
argument applies when $n$ is odd, so we have established (i).

The claim (ii) now follows from the vanishing ranges in Lemma \ref{lem-c2-vanish-range}.
\end{proof}

\begin{remark} We have actually shown something stronger. When $V = a+b\sigma$, the even gradations
split off after smashing with $\hf$ as long as $a-b \ge -1$, and the odd gradations split off after
smashing with $\hf$ as long as $a-b \le 0$.
\end{remark}
\begin{remark} The splitting of $\hf \wedge \mathbb{P}_2(S^{m\rho-\varepsilon})$
which is constructed inductively like this is actually canonical, for $\varepsilon =0,1$. That is,
at each step after the first, there is a unique $\hf$-linear homotopy section of 
$\hf \wedge F_n\mathbb{P}_2(S^{m\rho-\varepsilon})
\to \hf \wedge \textup{gr}_n\mathbb{P}_2(S^{m\rho-\varepsilon})$.
\end{remark}

In the next section we'll study the homology of specific extended powers, but before
we do there is one useful observation we can make in general that will be used later.

\begin{corollary}\label{cor-vanishing}
 When $V=a+b\sigma$, the Mackey functors $\hf_{s\rho+V}\mathbb{P}_2(S^{V})$ vanish for
$s< \min\{a, \frac{a+b}{2}\}$. 
\end{corollary} 
\begin{proof} It suffices to check the vanishing for the associated graded 
$\Sigma^{-V}\textup{gr}_*\mathbb{P}_2(S^V)$
where it follows, after some analytic geometry, from the vanishing ranges for the homology of a point.
\end{proof}

\section{Homology of Extended Powers}
It is well known that power operations for homology
 are parameterized by elements in $\hf_V\mathbb{P}_2(S^W)$. Unfortunately, we don't know
 the homology of general extended powers of representation spheres. On the other hand, if
 $W = a+b\sigma$ there is always a map
 	\[
	\Sigma^{a+b{\sigma}}\mathbb{P}_2(S^0) \to \mathbb{P}_2(S^{a+b\sigma})
	\]
which we might try to use to study the homology of the target. Asking that our power operations
behave well with respect to this suspension map is essentially equivalent to asking that they
factor through the inverse limit
	\[
	\holim_V \Sigma^{V}\hf \wedge \mathbb{P}_2(S^{-V}).
	\]
The representations $n\rho$ are cofinal amongst all representations, so we are free to
restrict attention to these. In \S3.1 we study the homotopy of the inverse limit above.

This inverse limit has another useful feature: we can identify it with a suspension
of a localization of the ring
$F(\textup{B}\mu_2, \hf)$. We carry this out in \S3.2 and then use this identification
to compute the Steenrod coaction on the homotopy of the inverse limit above. We will
use this later to study the relationship between power operations and Steenrod operations,
as well as to compute the action of the power operations on the dual Steenrod algebra.

\subsubsection*{Historical remarks} On the algebraic side, the relationship
between Tate-like constructions and operations goes back to Singer \cite{Sing}.
Jones and Wegmann realized this algebraic construction topologically in
\cite{JW} (see especially their results 5.2(a) and 5.2(b) and compare with our
\ref{ops-homotopy} and \ref{thm-comodule-structure}). More recent references
include course notes of Lurie \cite{Lur2} and a paper of Kuhn and McCarty
\cite[Cor. 2.13]{KM}.

\subsection{Homology, suspension, and the spectrum of operations}
We can re-write the splitting of the previous section as a calculation of homology groups.

\begin{theorem}\label{thm:homology-ext-powers} For $m \in \mathbb{Z}$, we have
	\[
	\hf_{\star}\mathbb{P}_2(S^{m\rho}) = \bigoplus_{s\ge m} \hf_{\star} \{e^{m\rho}_{s\rho}\} \oplus
	\bigoplus_{s\ge m} \hf_{\star}\{e^{m\rho}_{s\rho + \sigma}\}
	\]
	\[
	\hf_{\star}\mathbb{P}_2(S^{m\rho-1}) = \bigoplus_{s\ge m-1} \hf_{\star} \{e^{m\rho-1}_{s\rho+\sigma}\} 
	\oplus
	\bigoplus_{s\ge m} \hf_{\star}\{e^{m\rho-1}_{s\rho}\}
	\]
where the generator $e^V_W$ has degree $V+W$.
\end{theorem}

Notice that when $m=0$ we recover the Hu-Kriz calculation of $\hf^{\star}B\mu_2$. We can even
describe the whole ring, whose structure is just like the motivic case.

\begin{proposition}\label{coh-bmu}\footnote{This is essentially contained in the proof of Theorem 6.22
of \cite{HK}.} As an algebra over $\hf_{\star}$, we have:
	\[
	\hf^{\star}B\mu_2 = \hf^{\star}[\![c, b]\!]/ (c^2 = a_{\sigma}c + u_{\sigma}b).
	\]
\end{proposition}
\begin{proof} We just need to check the relation, which can be done in Borel cohomology. The Borel cohomology 
spectral sequence yields:
	\[
	(\hf^h)^{\star}B\mu_2 = (\hf^h)^{\star}[\![w]\!], \quad |w| = 1.
	\]
And, in particular:
	\[
	\left(\hf^h\right)^{\sigma}B\mu_2 = \underline{\F}\{u_{\sigma}w, a_{\sigma}\}.
	\]
We know that the restriction of $c$ to the underlying cohomology must be $u_{\sigma}w$ since there is only one generator in that degree. So, modulo $a_{\sigma}$, $c \mapsto u_{\sigma}w$. Since we have required
that $c$ maps to zero after restriction to a point, we can actually conclude that $c\mapsto u_{\sigma}w$.

Now recall that $(\hf^{hC_2})^* = \textup{H}^*(\mathbb{R}P^{\infty}, \mathbb{F}_2) = \mathbb{F}_2[\![x]\!]$ where
$x = \dfrac{a_{\sigma}}{u_{\sigma}}$. From the computation $\beta x = \textup{Sq}^1x = x^2$ we deduce
that 
	\[
	b \mapsto \beta(u_{\sigma}w) = u_{\sigma}w^2 + a_{\sigma}w.
	\]
The relation now follows from inspection.
\end{proof}
As we mentioned in the introduction to this section,
we're interested in power operations in an inverse limit involving maps
like
	\[
	\hf \wedge \Sigma^V \mathbb{P}_2(S^{-V}) \longrightarrow \hf \wedge \Sigma^W\mathbb{P}_2(S^{-W}),
	\]
where $V$ contains at least as many copies of each irreducible representation as $W$.  This
map comes from a diagonal
	\[
	S^V \wedge (S^{-V})^{\wedge 2} \longrightarrow S^W \wedge 
	\left(S^{V-W} \wedge S^{-V}\right)^{\wedge 2},
	\]
which is $(C_2\times \Sigma_2)$-equivariant. We are free to restrict to a cofinal system of representations,
so we'll focus on the inverse system:
	\[
\cdots \longrightarrow \hf \wedge \Sigma^{m\rho} \mathbb{P}_2(S^{-m\rho}) \longrightarrow
	\hf \wedge \Sigma^{(m-1)\rho +1}\mathbb{P}_2(S^{-(m-1)\rho -1})
	\longrightarrow \hf \wedge \Sigma^{(m-1)\rho}\mathbb{P}_2(S^{-(m-1)\rho})
	\longrightarrow \cdots.
	\]
\begin{definition} We define the {\bf spectrum of quadratic operations} by the homotopy limit
	\[
	\textup{Ops}_{\hf} = \holim_V \hf \wedge \Sigma^V \mathbb{P}_2(S^{-V}).
	\]
\end{definition}
\begin{warning} This spectrum is not the same as $\hf \wedge \holim_V \Sigma^V\mathbb{P}_2(S^{-V})$.
Indeed, it is a theorem of Lin that the spectrum
$\holim_n \Sigma^n\mathbb{P}_2(S^{-n})$ is $S^{-1}$.
\end{warning}
We want to compute the homotopy groups of $\textup{Ops}_{\hf}$. 
That means we need to study the 
effect on homology of the maps:
	\[
	\theta: \Sigma \mathbb{P}_2(S^{m\rho -1}) \longrightarrow \mathbb{P}_2(S^{m\rho}),
	\]
	\[
	\theta^{\sigma}: \Sigma^{\sigma}\mathbb{P}_2(S^{m\rho}) \longrightarrow \mathbb{P}_2(S^{m\rho + \sigma})
	=\mathbb{P}_2(S^{(m+1)\rho - 1}).
	\]
\begin{proposition} With notation as above, the effect of $\theta$ and $\theta^{\sigma}$ on homology is determined
by:
	\[
	\theta(e^{m\rho-1}_{s\rho}) = e^{m\rho}_{s\rho}
	, \quad \theta(e^{m\rho -1}_{s\rho+\sigma}) = 
	\begin{cases}
		e^{m\rho}_{s\rho+\sigma} & s\ge m\\
		0 & s=m-1.
	\end{cases}.
	\]
	\[
	\theta^{\sigma}(e^{m\rho}_{s\rho +\sigma}) = e^{(m+1)\rho -1}_{s\rho + \sigma}
	, \quad \theta^{\sigma}(e^{m\rho}_{s\rho}) = \begin{cases}
	e^{(m+1)\rho -1}_{s\rho} & s\ge m+1\\
	0 & s = m
	\end{cases}.
	\]
\end{proposition}
\begin{proof} It is clear from construction that the maps $\theta$ and $\theta^{\sigma}$ are filtration
preserving and so preserve this filtration degree on homology. In other words, $\theta$ and $\theta^{\sigma}$
induce maps $\hf \wedge S^V \longrightarrow \hf \wedge S^V$ and the question is whether generators
are sent to generators. Such maps are determined by an element in $\hf_0$, and all such elements
are determined by their restriction. So the effects of $\theta$ and $\theta^{\sigma}$ are determined
by the maps on \emph{underlying} spectra, where the answer is as indicated and known classically.
\end{proof}
\begin{remark} We don't know whether the map $\hf_{\star}\Sigma^V\mathbb{P}_2(S^W)
\to \hf_{\star}\mathbb{P}_2(S^{V+W})$ is surjective more generally.
\end{remark}
\begin{theorem}\label{ops-homotopy} The homotopy Mackey functors 
	$\rpis \textup{Ops}_{\hf}$
are given, as a module over $\hf_{\star}$, by 
	\[
	\rpis \textup{Ops}_{\hf} = \bigoplus_{s \in \mathbb{Z}} \hf_{\star}\{e_{s\rho}\} \oplus 
	\bigoplus_{s \in \mathbb{Z}} \hf_{\star}\{e_{s\rho +\sigma}\}
	\]
where the generator $e_V$ is in dimension $V$. 
\end{theorem}
\begin{proof} The proposition implies that the inverse system of homotopy Mackey functors is
Mittag-Leffler (since each map is surjective), and this almost gives the result. The caveat is
that the actual limit is a completion of the stated module which is determined by the
requirement
	\[
	\lim_{s \to -\infty} e_{s\rho} = \lim_{s \to -\infty}e_{s\rho +\sigma} = 0.
	\]
In other words, we need to rule out the possibility of an infinite sum converging to a weird
element in degree $a+b\sigma$. This can't happen because, if we fix $a$ and $b$,
then $\hf_{\star}\{e_{s\rho}\}$
vanishes in degree $a+b\sigma$ for sufficiently negative $s$, e.g. $s< \text{min}(a,b)$. 
\end{proof}
\begin{remark} The completion does arise in the geometric case, where $a_{\sigma}$ is inverted.
For example, there is the element
\[
\sum_{s\ge 0} \frac{u^s_{\sigma}}{a_{\sigma}^{2s}}e_{-s\rho}
\]
in degree 0.
\end{remark}
\subsection{Relationship to a Tate construction}
In the non-equivariant case, $\textup{Ops}_{\textup{H}\mathbb{F}_2}$ has the curious feature that 
$\Sigma (\textup{Ops}_{\textup{H}\mathbb{F}_2})$ is a ring spectrum and receives a ring map from
$F(\mathbb{R}P^\infty_+, \textup{H}\mathbb{F}_2)$. The best way to prove this is to identify the former
spectrum with the Tate spectrum $\hf^{t\Sigma_2}$. Here is a low-tech construction of this map and the ring structure.

First observe that the generator $x \in \textup{H}\mathbb{F}_2^1(\mathbb{R}P^1)$ defines a filtered
system of towers:
	\[
	\xymatrix{
	\{\textup{H}\mathbb{F}_2 \wedge D\mathbb{R}P_+^m\} \ar[r]^-{\cdot x} &
	\{\Sigma(\textup{H}\mathbb{F}_2 \wedge D\mathbb{R}P_+^{m+1})\} \ar[r]^-{\cdot x} &
	\{\Sigma^2(\textup{H}\mathbb{F}_2 \wedge D\mathbb{R}P_+^{m+2})\} \ar[r] & \cdots}
	\]
Atiyah duality and a Thom isomorphism allows one to compare with the filtered system of towers:
	\[
	\xymatrix{
	\{\textup{H}\mathbb{F}_2 \wedge \Sigma \mathbb{R}P^{-1}_{-m-1}\} \ar[r]&
	\{\textup{H}\mathbb{F}_2 \wedge \Sigma \mathbb{R}P^{0}_{-m-1}\} \ar[r] &
	\{\textup{H}\mathbb{F}_2 \wedge \Sigma \mathbb{R}P^1_{-m-1}\} \ar[r] & \cdots
	}
	\]
After checking some finiteness and vanishing conditions, it turns out you can take the homotopy limit
and homotopy colimit in either order. Doing it one way yields $F(\mathbb{R}P^{\infty}_+, \textup{H}\mathbb{F}_2)[x^{-1}]$,
while doing it the other gives $\Sigma (\textup{Ops}_{\textup{H}\mathbb{F}_2})$, and you're done. The
equivalence preserves the completed Steenrod coactions on both sides, so we can use
this to compute the coaction on $\pi_*\textup{Ops}_{\textup{H}\mathbb{F}_2}$.

We'd like to carry over this argument to the equivariant case.
\begin{theorem}\label{thm-tate-equivalence} The $C_2$-spectrum $\Sigma(\textup{Ops}_{\hf})$ receives a map
	\[
	\phi: F((B\mu_{2})_+, \hf) \longrightarrow \Sigma(\textup{Ops}_{\hf})
	\]
with the following properties:
	\begin{enumerate}[(a)]
	\item The induced map on $\rpis$ is a homomorphism of (complete)
	$\left(\arep\right)_{\star}$-comodules,
	\item $\phi$ factors through an equivalence
		\[
		F((B\mu_{2})_+, \hf)[b^{-1}] \stackrel{\cong}{\longrightarrow} \Sigma(\textup{Ops}_{\hf}),
		\]
	\item this equivalence takes $b^{-s}$ to the element $\Sigma e_{(s-1)\rho + \sigma}$ and
	$cb^{-s}$ to the element $\Sigma e_{(s-1)\rho}$.
	\end{enumerate}
\end{theorem}
\begin{remark} One might wonder what role the \emph{odd} steps in the filtration of $\textup{B}\mu_2$
have in the process of inverting $b$. The fact that multiplication by $b$ can be expressed as this
composite is essentially equivalent to the observation that $b$ is the Bockstein on the class $c$.
\end{remark}

The proof will require a few preliminaries. Up to now, we have avoided interpreting our
filtration on extended powers in terms of Thom spectra, as in \cite{T}. We can no
longer get away with this, since we'll need to identify Spanier-Whitehad duals via Atiyah duality.

\begin{definition} Let $\mathbb{R}P^{k\rho}$ and $\mathbb{R}P^{k\rho -1}$ denote 
$S((k\rho +1)\otimes \tau)/\Sigma_2$ and $S(k\rho \otimes \tau)/\Sigma_2$, respectively.
Note that $\hocolim_k \mathbb{R}P^{k\rho} = B\mu_{2+}$. 
\end{definition}

Given a $C_2 \times \Sigma_2$-representation, $W$, we get a $C_2$-equivariant vector
bundle on $\mathbb{R}P^{k\rho - \varepsilon}$ defined by 
	\[
	S((k\rho +1 - \varepsilon) \otimes \tau) \times_{\Sigma_2} W \longrightarrow 
	\mathbb{R}P^{k\rho -\varepsilon},
	\]
for $\varepsilon =0,1$. This extends to an assignment of a virtual bundle to a virtual representation.
We'll denote the Thom spectrum associated to some virtual representation $W$ by
	\[
	\left(\mathbb{R}P^{k\rho -\varepsilon}\right)^{W}.
	\]
Immediate from the geometric definition of our filtration on extended powers, we get the next lemma.
\begin{lemma}\label{thom-filtration} With notation as in (\ref{geometric-filtration}) and $\varepsilon=0,1$,
there are canonical equivalences
	\[
	F_{2j - \varepsilon}\mathbb{P}_2(S^{k\rho}) \cong 
	\left(\mathbb{R}P^{2j -\varepsilon}\right)^{k\rho(1+\tau)},
	\]
	\[
	F_{2j - \varepsilon}\mathbb{P}_2(S^{k\rho -1}) \cong 
	\left(\mathbb{R}P^{2j -\varepsilon}\right)^{(k\rho-1)(1+\tau)}.
	\]
\end{lemma}

We'll also need to know the duals of these Thom spectra.
The tangent bundle of $\mathbb{R}P^{k\rho -\varepsilon}$ is given by
$(k\rho -\varepsilon +1)\tau -1$, so equivariant Atiyah duality \cite{Wirth}
 implies the following proposition.

\begin{proposition}\label{atiyah-dual}
The Spanier-Whitehead dual of $\left(\mathbb{R}P^{k\rho -\varepsilon}\right)^{W}$
is given by
	\[
	D\left(\mathbb{R}P^{k\rho -\varepsilon}\right)^{W} \cong
	\left(\mathbb{R}P^{k\rho -\varepsilon}\right)^{1-(k\rho-\varepsilon+1)\tau - W}.
	\]	
for $\varepsilon=0,1$. In particular,
	\[
	D\left(\mathbb{R}P_+^{k\rho -\varepsilon}\right) \cong 
	\left(\mathbb{R}P^{k\rho -\varepsilon}\right)^{1-(k\rho-\varepsilon+1)\tau}.
	\]
\end{proposition}
Combining the previous two results, we get a canonical equivalence
	\begin{align*}
	D\mathbb{R}P_+^{k\rho} &\cong \left(\mathbb{R}P^{k\rho}\right)^{1-(k\rho +1)\tau}
	 = \Sigma^{k\rho +2}\left(\mathbb{R}P^{k\rho}\right)^{(-k\rho -1)(1+\tau)}
	\cong \Sigma^{k\rho +2}F_{2k}\mathbb{P}_2(S^{-k\rho-1}) & (\thetheorem.1)
	\end{align*}

Now we are interested in a certain filtered system of towers. The structure maps of the tower
are refinements of the maps $\theta$ and $\theta^{\sigma}$ we used earlier when
studying the suspension map. We'll denote by $\theta^{\rho}$ the composite
$\theta^{\sigma}\circ \theta$. 

\begin{lemma}\label{lift-suspension} The following diagram commutes:
	\[
	\xymatrix{
	D\mathbb{R}P_+^{(k+1)\rho} \ar[r]^{\textup{incl}^{\vee}}
	\ar@{-}[d]_{\cong}& D\mathbb{R}P_+^{k\rho}\ar@{-}[d]^{\cong} \\
	\Sigma^{(k+1)\rho +2}F_{2k+2}\mathbb{P}_2\left(S^{-(k+1)\rho -1}\right)
	\ar[dr]_-{\theta^{\rho}}
	& \Sigma^{k\rho +2}F_{2k}\mathbb{P}_2\left(S^{-k\rho -1}\right)\ar[d]^{\textup{incl}}\\
	&\Sigma^{k\rho +2}F_{2k+2}\mathbb{P}_2\left(S^{-k\rho -1}\right)
	}
	\]
\end{lemma} 
\begin{proof} Recall that, given a closed, equivariant inclusion $N \hookrightarrow M$ with
normal bundle $\nu$ and
a virtual equivariant vector bundle $W$ on $M$, there is a twisted Pontryagin-Thom collapse
	\[
	M^W \longrightarrow N^{W|_N \oplus \nu}.
	\]
When $W = -TM$ is the stable normal bundle, then the collapse has the form 
	\[
	\textup{coll}: M^{-TM} \longrightarrow N^{-TM|_N \oplus \nu}
	\cong N^{-TN}.
	\]
This models the dual $DM_+ \longrightarrow DN_+$
of the original inclusion. In our case this translates into the commutativity of the diagram
	\[
	\xymatrix{
	D\mathbb{R}P_+^{(k+1)\rho} \ar[r]^{\textup{incl}^{\vee}}
	\ar@{-}[d]_{\cong}& D\mathbb{R}P_+^{k\rho}\ar@{-}[d]^{\cong}\\
	\left(\mathbb{R}P^{(k+1)\rho}\right)^{1-((k+1)\rho +1)\tau}\ar[r]_-{\textup{coll}}
	&\left(\mathbb{R}P^{k\rho}\right)^{1-(k\rho +1)\tau}.
	}
	\]
Now, under the identification of Lemma \ref{thom-filtration}, the map $\theta^{\rho}$ corresponds
to the map of Thom spectra induced by adding $\rho\tau$ to the stable normal bundle. So we are
left with checking the commutativity of:
	\[
	\xymatrix{
	\left(\mathbb{R}P^{(k+1)\rho}\right)^{1-((k+1)\rho +1)\tau}\ar[dr]_-{\oplus \rho\tau}\ar[r]^-{\textup{coll}}
	&\left(\mathbb{R}P^{k\rho}\right)^{1-(k\rho +1)\tau}\ar[d]^-{\textup{incl}}\\
	& \left(\mathbb{R}P^{(k+1)\rho}\right)^{1-(k\rho +1)\tau}
	}
	\]
By naturality of collapse, this is equivalent to the commutativity of:
	\[
	\xymatrix{
	\left(\mathbb{R}P^{(k+1)\rho}\right)^{1-((k+1)\rho +1)\tau}\ar[dr]^-{\oplus \rho\tau}\ar[d]_-{\textup{incl}}
	&\\
	\left(\mathbb{R}P^{(k+2)\rho}\right)^{1-((k+1)\rho +1)\tau}\ar[r]_-{\textup{coll}}
	& \left(\mathbb{R}P^{(k+1)\rho}\right)^{1-(k\rho +1)\tau}
	}
	\]
Finally, this diagram commutes by the definition of the collapse together with the fact that the normal
bundle of the inclusion $\mathbb{R}P^{(k+1)\rho} \hookrightarrow \mathbb{R}P^{(k+2)\rho}$ is
$\rho\tau$. 
\end{proof}

The above lemma ensures that the dual of the inclusions of projective spaces gives
a lift of $\theta^{\rho}$ which we denote by $\widetilde{\theta}^{\rho}$. More generally,
we'll denote by $\widetilde{\theta}^{j\rho}$ the $j$-fold composite of $\widetilde{\theta}^{\rho}$,
or any of its suspensions.

\begin{construction} Define a functor $\textup{Ops}^*_*: \mathbb{Z}_{\ge 0}^{op} \times \mathbb{Z}_{\ge 0}
\longrightarrow \textup{Sp}^{C_2}$ on objects by
	\[
	\textup{Ops}^i_{j} := \Sigma^{i\rho +1}F_{2(i+j)}\mathbb{P}_2\left(S^{-i\rho -1}\right),
	\]
and on morphisms $(i, j) \rightarrow (k, \ell)$ by:
	\[
\textup{incl} \circ \widetilde{\theta}^{(k-i)\rho}:\Sigma^{i\rho +1}F_{2(i+j)}\mathbb{P}_2\left(S^{-i\rho -1}\right)
\longrightarrow
\Sigma^{k\rho +1}F_{2(k+\ell)}\mathbb{P}_2\left(S^{-k\rho -1}\right).
	\]
\end{construction}

By design:
	\[
	\textup{Ops}_{\hf} \cong \holim_i \hocolim_j \left(\hf \wedge \textup{Ops}^i_j\right). 
	\]
	
The next thing to do is build a map of towers
	\[
	\{D\mathbb{R}P_+^{k\rho}\} \longrightarrow \left\{\Sigma \left(\textup{Ops}^k_0\right)\right\}.
	\]
By (2.26.1)
, and our definition of $\textup{Ops}_0^k$, we have a canonical equivalence 
$D\mathbb{R}P_+^{k\rho} \cong 
\Sigma \left(\textup{Ops}^k_0\right)$. This gives:
	\[
	\hf \wedge D\mathbb{R}P_+^{k\rho} \longrightarrow \hocolim_j \Sigma \hf \wedge \textup{Ops}^k_j.
	\]
By our definition of $\widetilde{\theta}^{\rho}$, these maps fit into a map of towers and yield:
	\[
	\phi: F(B\mu_{2+}, \hf)\longrightarrow \Sigma \left(\textup{Ops}_{\hf}\right).
	\] 
The
induced homomorphism on $\rpis$ respects the comodule structure since we started with
maps of spectra.

\begin{lemma}\label{b-to-e}
The map $\phi: F(B\mu_{2+}, \hf) \longrightarrow \Sigma\left(\textup{Ops}_{\hf}\right)$
has the following effect on homotopy, for $s\le 0$ and $\varepsilon=0,1$:
	\[
	\phi_*c^{\varepsilon}b^{-s} = \Sigma e_{(s-1)\rho+\varepsilon\sigma}.
	\]
\end{lemma}
\begin{proof} This is a restatement of the fact that $c^{\varepsilon}b^{-s}$ is dual
to $e^0_{s\rho+\varepsilon\sigma}$, using the notation of Theorem \ref{thm:homology-ext-powers}.
\end{proof}
The remainder of this section will be devoted to showing that this map
identifies the target with $F(B\mu_{2+}, \hf)[b^{-1}]$. 

First we need to interchange the homotopy limit and the homotopy colimit.

\begin{lemma} The canonical map
	\[
	\hocolim_j\holim_i \hf \wedge \textup{Ops}^i_j \longrightarrow \holim_i\hocolim_j \hf \wedge
	\textup{Ops}^i_j
	\]
is an equivalence.
\end{lemma}
\begin{proof} We showed earlier that the inverse systems of homology groups are Mittag-Leffler,
so we are reduced to showing that the map
	\[
	\colim_j \lim_i \hf_{\star}\textup{Ops}^i_j \longrightarrow \lim_i \colim_j \hf_{\star}\textup{Ops}^i_j
	\]
is an equivalence. This is immediate from the calculation of $\rpis \textup{Ops}_{\hf}$
(\ref{ops-homotopy}) and
the fact that the elements $e_{s\rho}$ and $e_{s\rho +\sigma}$ are detected in 
$\hf_{\star}\textup{Ops}^i_j$ for $i$ and $j$ sufficiently large. 
\end{proof}

Finally, we identify the homotopy limit of each tower with a suspension of a dual projective space via
a Thom isomorphism and determine the resulting map. The main theorem of the section follows.

\begin{proposition} There is an equivalence 
	\[
	\holim_i\hf \wedge \Sigma \left(\textup{Ops}^i_j\right) \cong \Sigma^{j\rho} F(B\mu_{2+}, \hf)
	\]
such that the composite:
	\[
	F(B\mu_{2+}, \hf) \longrightarrow 
	\holim_i\hf \wedge \Sigma \left(\textup{Ops}^i_j\right) \cong \Sigma^{j\rho} F(B\mu_{2+}, \hf)
	\]
induces multiplication by $b^j$ on homotopy.
\end{proposition}
\begin{proof} We need to check that the following diagram commutes:
	\[
	\xymatrix{
	\rpis\hf \wedge D\mathbb{R}P_+^{k\rho} \ar[rr]^-{\cdot b}\ar[d]_-{\cong}
	&& \rpis\Sigma \hf \wedge D\mathbb{R}P_+^{k\rho+1}\ar[d]^-{\cong}\\
	\rpis\hf \wedge \Sigma^{k\rho +2}F_{2k}\mathbb{P}_2(S^{-k\rho-1}) \ar[r]^-{\textup{incl}}&
	\rpis\hf \wedge \Sigma^{k\rho +2}F_{2k+2}\mathbb{P}_2(S^{-k\rho-1}) \ar[r]^-{\cong}&
	\rpis\hf \wedge \Sigma^{k\rho+4}F_{2k+2}\mathbb{P}_2(S^{-(k+1)\rho -1})
	}
	\]
where the bottom horizontal isomorphism comes from the splitting of Theorem 
\ref{thm-extended-power-filtration}. Specifically, this isomorphism takes 
$e^{-k\rho-1}_{s\rho +\varepsilon\sigma}$
to $e^{-(k+1)\rho -1}_{(s+1)\rho+\varepsilon\sigma}$ where $\varepsilon=0, 1$. Commutativity
now follows from Lemma \ref{b-to-e}, since this implies that the two ways around the diagram
agree on generators.

\end{proof}

\subsection{Comodule structure}
We can now determine the completed comodule structure on 
$\rpis \textup{Ops}_{\hf}$. It will be helpful to introduce some notation.
\begin{notation} Let $t$ be a formal variable of degree $-\rho$, and $dt$ a variable of
degree $-1$. Then set
	\[
	\xi(t) = \sum_{i\ge 0} \xi_i t^{2^i}, \quad \tau(t) = \sum_{i\ge 0} \tau_it^{2^i}
	\]
so that $|\xi(t)| = |\tau(t) dt| = -\rho$. Similarly define
	\[
	\overline{\xi}(t) = \sum_{i\ge 0} \overline{\xi}_it^{2^i}, \quad
	\overline{\tau}(t) = \sum_{i\ge 0}\overline{\tau}_it^{2^i}.
	\]
Given a formal power series $f(t)$ we use $\left[ f(t)\right]_{t^r}$ to denote the coefficient of $t^r$.
\end{notation}

\begin{theorem}\label{thm-comodule-structure}
 The completed, left $\arep$-comodule structure on
$\rpis \textup{Ops}_{\hf}$ is given by the following formulas:
	\begin{align*}
	\psi_L(e_{s\rho + \sigma}) = \sum_{r \in \mathbb{Z}} \left[ \xi(t)^r\right]_{t^s} \otimes e_{r\rho + \sigma},\\
	\psi_L(e_{s\rho}) = \sum_{r \in \mathbb{Z}} \left[\xi(t)^r\right]_{t^s} \otimes e_{r\rho} +
	 \sum_{r \in \mathbb{Z}} \left[\xi(t)^r\tau(t)\right]_{t^s} \otimes e_{r\rho + \sigma}.
	\end{align*}
\end{theorem}

Before proving this let's collect a few preliminaries.

\begin{lemma}\label{lem-comodule-structure} The right, completed $\arep$-comodule structure on
$\rpis F((B\mu_2)_{+}, \hf)[b^{-1}]$ is given, for $s \in \mathbb{Z}$, by:
	\begin{align*}
	\psi_R(b^s) = \sum_{i \in \mathbb{Z}} b^i \otimes \left[\xi(t)^s\right]_{t^i},\\
	\psi_R(cb^s) = \sum_{i \in \mathbb{Z}} cb^i \otimes \left[\xi(t)^s\right]_{t^i}+
	\sum_{i \in \mathbb{Z}} b^i \otimes \left[\tau(t)\xi(t)^s\right]_{t^i}.
	\end{align*}
\end{lemma}
\begin{proof} We already know these formulas hold for $\psi_R(c)$ and $\psi_R(b)$, by
the definition of $\xi_i$ and $\tau_i$, and the general
case follows since $\psi_R$ is a map of algebras.
\end{proof}
Changing from left to right comodule structures introduces a conjugate, so we will need
to identify certain coefficients of, say, $\overline{\xi}(t)^s$. This is the content of the following
two lemmas.

\begin{lemma}\label{lem-id} We have the following identities of formal power series:
	\begin{align*}
	\xi(\overline{\xi}(t)) = \overline{\xi}(\xi(t)) = t\\
	\tau(\overline{\xi}(t)) = \overline{\tau}(t)\\
	\overline{\tau}(\xi(t)) = \tau(t).
	\end{align*}
\end{lemma}
\begin{proof} These are restatements of the identities:
	\[
	\tau_i + \sum \xi^{2^j}_{i-j}\overline{\tau}_j = \sum \xi_{i-j}^{2^j}\overline{\xi}_j = 0,
	\]
which follow from the formula for comultiplication and the definition of conjugation.
\end{proof}
\begin{lemma}\label{lem-residue} We have
	\begin{align*}
	\left[\overline{\xi}(t)^{-s-1}\right]_{t^{-r-1}} = \left[\xi(t)^r\right]_{t^s}\\
	\left[\overline{\tau}(t)\overline{\xi}(t)^{-s-1}\right]_{t^{-r-1}} = \left[\xi(t)^r\tau(t)\right]_{t^s}.
	\end{align*}
\end{lemma}
\begin{proof} We will prove the second identity- the first one is easier. We'll use $\oint (-) dz$
to denote the operation of taking the coefficient of $z^{-1}$ in a formal Laurent series. Then
we compute
	\begin{align*}
	\left[\overline{\tau}(t)\overline{\xi}(t)^{-s-1}\right]_{t^{-r-1}} &=
	\oint t^r\,\overline{\tau}(t)\overline{\xi}(t)^{-s-1} dt \\
	&= \oint \xi(u)^r\tau(u)u^{-s-1}\xi'(u) du & \text{ (change of variables and 
	(\ref{lem-id}))}\\
	&= \oint \xi(u)^r\tau(u)u^{-s-1}du & (\xi'(u) =1\text{ because } 2=0)\\
	&= \left[\xi(u)^r\tau(u)\right]_{u^s}.
	\end{align*}
\end{proof}
\begin{proof}[Proof of Theorem \ref{thm-comodule-structure}] Identifying $\Sigma e_{s\rho}$ with
$cb^{-s-1}$ and $\Sigma e_{s\rho + \sigma}$ with $b^{-s-1}$ using Theorem \ref{thm-tate-equivalence},
we are reduced by
Lemma \ref{lem-comodule-structure} to Lemma \ref{lem-residue}.
\end{proof}
\begin{remark}\begin{enumerate}[(i)]
\item This section is inspired by
the description of the comodule structure on truncated projective spaces given
in \cite{Ba}. (Though the actual derivation of the formulas was not given there, so it's
possible the proof above is not the most efficient.)
\item The notation is meant to suggest a relationship between the dual Steenrod
algebra and the Hopf algebroid of automorphisms of some kind of `equivariant, formal,
additive supergroup' corresponding to the addition map $B\mu_2 \times B\mu_2
\longrightarrow B\mu_2$. Compare \cite{I}, for example. The situation here is
complicated by the fact that $\hf^{\star}B\mu_2$ isn't the trivial square
zero extension of $\underline{\F}[b]$, amongst other things.
\end{enumerate}
\end{remark}

\section{Power Operations}

The definition of power operations we give below is slightly nonstandard. It appears
explicitly
in the non-equivariant context, for example,
in \cite[(2.2.6)]{Lur}.

\subsection{Definition and statement of main theorem}
Recall that we can define extended powers over any $C_2$-equivariant commutative ring spectrum.
In particular, if $A$ is an $\hf$-module, we can make sense of the definition
	\[
	\mathbb{P}_2^{\F}(A) := \left(E\mu_{2+} \wedge_{\hf} A^{\wedge_{\hf}2}\right)/\Sigma_2. 
	\]
In the case $A = \hf \wedge X$ for some $C_2$-spectrum $X$, this simplifies (\ref{lem-change-base}):
	\[
	\mathbb{P}_2^{\F}(\hf \wedge X) \cong \hf \wedge \mathbb{P}_2(X). 
	\]
Just as before, there's a map $\Sigma^V \mathbb{P}^{\mathbb{F}}_2(A) \to \mathbb{P}_2^{\F}(\Sigma^VA)$
and we can define a homotopical functor
	\[
	\textup{Ops}_{\hf}(-): \textup{Mod}_{\hf} \longrightarrow \textup{Mod}_{\hf}
	\]
such that $\textup{Ops}_{\hf}(\hf) = \textup{Ops}_{\hf}$. 

\begin{construction} The functors $\Sigma_+^{\infty}$ and $\textup{Ops}_{\hf}(-)$ induce morphisms of derived 
mapping spaces, natural in a space $X$:
	\[
	X \to \textup{Map}_{\textup{Spaces}}(\ast, X) \to \textup{Map}_{\textup{Spectra}}(S^0, \Sigma^{\infty}_+X)
	\to \textup{Map}_{\F}(\hf, \hf \wedge \Sigma^{\infty}_+X)
	\to \textup{Map}_{\F}(\textup{Ops}_{\hf}, \textup{Ops}_{\hf}(\hf \wedge \Sigma^{\infty}_+X)).
	\]
By adjunction we get a natural map
	\[
	\alpha: \textup{Ops}_{\hf} \wedge_{\hf}(\hf \wedge  \Sigma^{\infty}_+X) \to 
	\textup{Ops}_{\hf}(\hf \wedge \Sigma^{\infty}_+X).
	\]
There is also a natural map $\Sigma^{-V}\textup{Ops}_{\hf}(A) \to \textup{Ops}_{\hf}(\Sigma^VA)$,
coming from the map on extended powers used to construct the functor $\textup{Ops}_{\hf}(-)$. Since every
$\hf$-module $A$ is a filtered colimit of modules of the form $\Sigma^{-V}\hf\wedge \Sigma^{\infty}_+X$,
we can extend our definition of $\alpha$ to an {\bf assembly map} for $A$:
	\[
	\alpha: \textup{Ops}_{\hf} \wedge_{\hf} A \longrightarrow \textup{Ops}_{\hf}(A).
	\]
\end{construction}

We can define power operations in a fairly weak setting.

\begin{definition} Let $R$ be an equivariant commutative ring.
We say that an $R$-module $A$ has an {\bf equivariant symmetric multiplication}
if it is equipped with $R$-linear maps $\mathbb{P}^R_2(A) \to A$ and $R \to A$ such that
	\[
	\xymatrix{
	\mathbb{P}^R_2(R) \ar[r]\ar[d] & \mathbb{P}^R_2(A) \ar[d]\\
	R \ar[r] & A
	}
	\]
commutes up to homotopy.
\end{definition}

\begin{definition} Let $A$ be an $\hf$-module with an equivariant
symmetric multiplication, and let 
$e \in \underline{\pi}_V\textup{Ops}_{\hf}$ be represented by some map
$\hf \rightarrow \Sigma^{-V}\textup{Ops}_{\hf}$ of right $\hf$-modules. Then define the power
operation $Q^e: A \rightarrow \Sigma^{-V}A$ by the composite:
	\[
	Q^e: A \cong \hf \wedge_{\hf} A \longrightarrow \Sigma^{-V}\textup{Ops}_{\hf} \wedge_{\hf} A
	\stackrel{\alpha}{\longrightarrow}
	 \Sigma^{-V}\textup{Ops}_{\hf}(A) \longrightarrow \Sigma^{-V}\mathbb{P}^{\uF}_2(A)
	\longrightarrow \Sigma^{-V}A.
	\]
In the special case when $e=e_{s\rho}$ or $e=e_{s\rho -1}$ we denote the resulting operations
by $Q^{s\rho}$ and $Q^{s\rho -1}$, respectively. We will also use the same notation for the
induced map on homotopy groups.
\end{definition}

The remainder of this section is devoted to proving the first
properties of these operations $Q^{k\rho -\varepsilon}$. For convenience we collect them all in the next
theorem. Throughout, $\varepsilon = 0,1$.
\begin{theorem}\label{omnibus} The operations $Q^{k\rho - \varepsilon}$ on the homotopy groups of
$\hf$-modules with an equivariant symmetric
multiplication satisfy the following properties.
\begin{enumerate}[\textup{(}a\textup{)}]
\item (Vanishing) If $x \in \rpi_{a+b\sigma} A$, then $Q^{s\rho}x =0$ for all $s<\min\left\{a, \dfrac{a+b}{2}\right\}$.
In
particular, if $x \in \rpi_{n\rho}A$ then $Q^{s\rho}x=0$ for $s < n$.
\item (Squaring) If $x \in \rpi_{n\rho-\varepsilon} A$, then $Q^{n\rho-\varepsilon}x = x^2$.
\item (Bockstein) Let $\beta$ denote the Bockstein associated to $\underline{\mathbb{Z}}
\stackrel{\cdot 2}{\to} \underline{\mathbb{Z}}
\to \underline{\F}$. Then, for all $s$,  $\beta Q^{s\rho} = Q^{s\rho -1}$. 
\item (Product formula) For any $x \in \rpis A$ and $y \in \rpis B$ we have
	\[
	Q^{k\rho}(x \wedge y) = \sum_{i+j = k} Q^{i\rho}x \wedge Q^{j\rho} y \in \rpis (A \wedge_{\hf} B).
	\]
\item (Additivity) The operations $Q^{s\rho}$ commute with addition, restriction, and transfers.
\item (Action on $\hf_{\star}$) $Q^0$ acts by the identity on $\hf_{\star}$, and $Q^{-1}u_{\sigma} = a_{\sigma}$.
\item (Unstable condition) If $X$ is a space, then $Q^0$ acts by the identity on $\rpis F(X, \hf)$ and
$Q^{k\rho}$ acts by zero for $k>0$. 
\item (Co-Nishida relations)  Let $E$ be
a $C_2$-spectrum with an equivariant
symmetric multiplication and let $x \in \hf_{\star}E$, then, for any $s \in \mathbb{Z}$,
we have an identity of formal power series:
	\[
	\sum_{r \in \mathbb{Z}} \psi_R\left(Q^rx\right)t^r =
	\sum_{r \in \mathbb{Z}} Q^{r\rho}(\psi_R(x)) \overline{\xi}(t)^r + 
	\sum_{r \in \mathbb{Z}} Q^{r\rho+\sigma}(\psi_R(x)) \overline{\xi}(t)^r\overline{\tau}(t).
	\]
\end{enumerate}
\end{theorem}

\begin{remark} It is not true that analogues of the Adem relations hold in this generality.
For such relations to hold, we need the extra structure of a map $\mathbb{P}_4^R(A)
\to A$ compatible with the equivariant symmetric multiplication. Since we have not
examined higher extended powers in this paper, the Adem relations will have to wait
for another time.
\end{remark}

We'll prove each of these below, except additivity which is trivial from the definition: the
operation $Q^{s\rho}$ on homotopy is induced by a map of spectra.

\subsection{Proofs of first properties}

\begin{lemma}[Vanishing] If 
$x \in \rpi_{a+b\sigma} A$, then $Q^{s\rho}x =0$ for all $s<\min\left\{a, \dfrac{a+b}{2}\right\}$. In
particular, if $x \in \rpi_{n\rho}A$ then $Q^{s\rho}x=0$ for $s < n$.
\end{lemma}
\begin{proof} Choose a map $S^{a+b\sigma} \longrightarrow A$ representing $x$. Then we get
a diagram
	\[
	\xymatrix{
	S^0 \wedge S^{a+b\sigma}\ar[r]^-{e_{s\rho}\wedge \textup{id}}\ar[d]_{x}
	&\Sigma^{-s\rho}\textup{Ops}_{\hf} \wedge S^{a+b\sigma}\ar[d]\ar[r]
	&\Sigma^{-s\rho}\hf \wedge \mathbb{P}_2(S^{a+b\sigma})\ar[d]
	&
	\\
	A \ar[r]
	& \Sigma^{-s\rho}\textup{Ops}_{\hf} \wedge_{\hf} A\ar[r]
	&\Sigma^{-s\rho}\mathbb{P}_2^{\F}(A) \ar[r]
	&\Sigma^{-s\rho}A
	}
	\]
The longest composite is $Q^{s\rho}x$, by definition, so a sufficient condition
for its vanishing is that the top horizontal composite vanishes. But the top
horizontal composite is an element in $\hf_{s\rho + a+b\sigma}\mathbb{P}_2(S^{a+b\sigma})$.
By Lemma \ref{cor-vanishing}, this group is zero in the range indicated.
\end{proof}
\begin{remark} Perhaps a better mnemonic is that the operations vanish on $x$ for ``$s< \frac{|x|}{\rho}$'',
in analogy with the odd primary Dyer-Lashof operations.
\end{remark}
\begin{lemma}[Squaring] Let $\varepsilon = 0,1$. For $x \in \underline{\pi}_{n\rho-\varepsilon}A$, 
$Q^{n\rho-\varepsilon}x = x^2 \in \underline{\pi}_{2n\rho-2\varepsilon}A$.
\end{lemma}
\begin{proof} The inclusion $\Sigma_{2+} = F_0E\mu_{2+} \hookrightarrow E\mu_{2+}$ induces a diagram
	\[
	\xymatrix{
	\hf \wedge \left(S^{n\rho-\varepsilon} \wedge S^{n\rho-\varepsilon}\right)\ar[r]\ar[d]
	& \hf \wedge \mathbb{P}_2\left(S^{n\rho-\varepsilon}\right)\ar[d]&\\
	A\wedge_{\F}A \ar[r] & \mathbb{P}_2^{\F}(A) \ar[r] & A
	}
	\]
The bottom horizontal composite is by definition the multiplication on the $N_{\infty}$-algebra $A$, and
the top horizontal arrow is the element $e_{n\rho-\varepsilon}^{n\rho-\varepsilon}$, again by definition.
The result follows.
\end{proof}
\begin{remark} The reason why we only get a squaring result in these degrees is because
our distinguished power operations are related to a homology basis for the extended powers
of spheres like $S^{n\rho -\varepsilon}$. We do not know the homology of
extended powers of arbitrary representation 
spheres $S^V$, but this bottom class is certainly there so one could define an operation
$Q^V$. It's unclear whether this class lifts to an element in $\rpis \textup{Ops}_{\hf}$, but
if it did then the operations $Q^V$ would be a linear combination of the operations
we've already listed.
\end{remark}
\begin{remark} Had we started our filtration of $E\mu_{2}$ with 
$\dfrac{C_2 \times \Sigma_2}{\Delta}$ instead, then the bottom operation would be the \emph{norm}
instead of the square.
\end{remark}
\begin{lemma}[Bockstein] For all $s \in \mathbb{Z}$, $\beta Q^{s\rho} = Q^{s\rho -1} = Q^{(s-1)\rho + \sigma}$.
\end{lemma}
\begin{proof} Recall that $\beta c = b \in \hf^{\rho}B\mu_{2}$. Since $\beta^2 = 0$ and $\beta$ is
a derivation on cohomology, it follows that $\beta(cb^{-{s+1}}) = b^{-s}$ in $\hf^{\star}(B\mu_{2})[b^{-1}]$.
By Theorem \ref{thm-tate-equivalence}
we conclude that $\beta e_{s\rho} = e_{s\rho -1}$ and the result follows from naturality of the Bockstein
and the definition of the power operations. 
\end{proof}

The universal example of power operations applied to a product of classes
in degree $V$ and $W$ is captured by a diagonal map
	\[
	\delta: \mathbb{P}_2(S^V \wedge S^W) \longrightarrow \mathbb{P}_2(S^V) \wedge \mathbb{P}_2(S^W).
	\]
Concretely, this arises from the diagonal of the space $E\mu_{2}$, which fits into a composite:
	\[
	E\mu_{2+} \wedge (S^V \wedge S^W)^{\wedge 2}
	\longrightarrow E\mu_{2+} \wedge E\mu_{2+} \wedge (S^V \wedge S^W)^{\wedge 2}
	\cong \left(E\mu_{2+} \wedge (S^V)^{\wedge 2}\right) \wedge \left(E\mu_{2+} \wedge
	(S^W)^{\wedge 2} \right).
	\]
Interpreting the source and target as Thom spectra, this diagonal arises from the diagram:
	\[
	\xymatrix{
	(V \oplus W)(1\oplus \tau) \ar[r]\ar[d] & V(1+\tau) \boxplus W(1+\tau)\ar[d]\\
	B\mu_2 \ar[r]_{\Delta} & B\mu_2 \times B\mu_2
	}
	\]
where we've used $\boxplus$ to denote an exterior sum. The next proposition follows.

\begin{proposition} The map
	\[
	\delta_*: \hf_{\star}\mathbb{P}_2(S^{k\rho} \wedge S^{\ell\rho})
	\longrightarrow \hf_{\star}\mathbb{P}_2(S^{k\rho}) \otimes_{\hf_{\star}}
	\hf_{\star}\mathbb{P}_2(S^{\ell\rho})
	\]
is given by the following formula together with the known action of the Bockstein:
	\[
	\delta_*(e^{(k+\ell)\rho}_{r\rho}) = \sum_{s+t=r} e^{k\rho}_{s\rho} \otimes e^{\ell\rho}_{t\rho}.
	\]
\end{proposition}
The diagonal maps taken together yield a map
	\[
	\delta: \textup{Ops}_{\hf} \longrightarrow \holim_{k, \ell} \Sigma^{(k+\ell)\rho}\hf \wedge 
	\mathbb{P}_2(S^{-k\rho})
	\wedge \mathbb{P}_2(S^{-\ell\rho})
	\]
and our computations of the homotopy of extended powers in this tower show that the canonical
map
	\[
	\textup{Ops}_{\hf} \wedge_{\hf} \textup{Ops}_{\hf} \longrightarrow 
	\holim_{k, \ell} \Sigma^{(k+\ell)\rho}\hf \wedge 
	\mathbb{P}_2(S^{-k\rho})
	\wedge \mathbb{P}_2(S^{-\ell\rho})
	\]
is an equivalence. So we get the following corollary.

\begin{corollary} The map
	\[
	\delta_*: \rpis \textup{Ops}_{\hf} \longrightarrow \rpis \textup{Ops}_{\hf} \otimes_{\hf_{\star}}
	\rpis \textup{Ops}_{\hf}
	\]
is given by
	\[
	\delta_*(e_{k\rho}) = \sum_{i+j=k} e_{i\rho} \otimes e_{j\rho}.
	\]
\end{corollary}

The next proposition then follows from a diagram chase from the definitions of our power operations.

\begin{proposition}[Product formula] For any $x \in \rpis A$ and $y \in \rpis B$ we have
	\[
	Q^{k\rho}(x \wedge y) = \sum_{i+j = k} Q^{i\rho}x \wedge Q^{j\rho} y \in \rpis (A \wedge_{\hf} B).
	\]
\end{proposition}

\begin{remark} The fact that $x$ and $y$ are permitted to be in an arbitrary degree is a nice consequence of the
use of $\textup{Ops}_{\hf}$ in our definition of power operations. This spectrum keeps track of
all the relations implied by the interactions between power operations and $RO(C_2)$-graded suspensions,
which allows us to forego knowledge of the homology of $\mathbb{P}_2(S^V)$ for arbitrary $V$.
\end{remark}

In order to compute the action of our power operations on $\hf_{\star}S^0$, we'll need to understand the
induced map on homology of
	\[
	\mathbb{P}_2(S^0) \longrightarrow S^0.
	\]
Algebraically, this is some additive map
	\[
	\bigoplus \hf_{\star}\{e^0_{s\rho}\} \oplus \bigoplus \hf_{\star}\{e^{0}_{s\rho + \sigma}\}
	\to \hf_{\star}. 
	\]
But, for degree reasons, all the generators must map to zero except possibly $e^0_0$.
On the span of this generator, we get the identity, which can be seen from the fact that squaring
$S^0 \to S^0 \wedge S^0 \to S^0$
factors through the extended power. 

\begin{remark} If we had used $\textup{H}\underline{R}$ for some $\F$-algebra, $R$, then
$e^0_0$ would map to $1$, but $re^0_0$ would map to $r^2$ for $r \in R$.
\end{remark}

Together with vanishing, and the Bockstein lemma, we get

\begin{lemma}[Action on $\hf_{\star}$] 
$Q^0$ acts by the identity on $\hf_{\star}$, and $Q^{-1}u_{\sigma} = a_{\sigma}$.
\end{lemma}

Finally, we turn to the action of power
operations on the cohomology of a $C_2$-space.
First we record a fact which is immediate from the definition of the power operations.

\begin{lemma} If $X$ is a $C_2$-space, then the action of power operations on $\rpis F(X_+, \hf)$
commutes with the suspension isomorphism. 
\end{lemma}

\begin{proposition}[Unstable condition]
If $X$ is a $C_2$-space, then $Q^0$ acts by the identity on $\rpis F(X_+, \hf)$ and
$Q^{k\rho}$ acts by zero for $k>0$.
\end{proposition}
\begin{proof} Using the suspension isomorphism, it suffices to check this result
for elements in degree $\hf^{V}(X)$ where $V$ is a representation containing
at least two copies of the trivial representation. By naturality, we may also
assume that $X = K(\F, V)$ and the cohomology class in question is the
universal one. The vanishing result now follows from the equivariant Hurewicz
theorem (which is the source of our assumption on $V$) applied to the fundamental map
	\[
	S^V \longrightarrow K(\F, V).
	\]
Also by the Hurewicz theorem we are reduced to checking the claim about $Q^0$ on
the fundamental class in $S^V$. By the suspension isomoprhism, this reduces to
the action of $Q^0$ on $1 \in \hf^{0}S^0=\hf_0S^0$. We have already seen that
$Q^01 = 1$, so we're done. 
\end{proof}

\section{The action of power operations on the dual Steenrod algebra}

Consider the following diagram of adjoints between homotopy theories:
	\[
	\xymatrix{
	N_{\infty}\textup{-Alg}_{\hf}\ar@{-}[d]\ar@{-}[r] & N_{\infty}\textup{-Alg}\ar@{-}[d]\\
	\textup{Mod}_{\hf}\ar@{-}[r] & \textup{Sp}^{C_2}
	}
	\]
Objects in the top right corner have two equivalent descriptions in terms of $\textup{Mod}_{\hf}$
with extra structure. Specifically, let $C$ denote the comonad on $\textup{Mod}_{\hf}$
whose comodules are ordinary spectra,
and let $\mathbb{P}^{\F}$ denote the monad on $\textup{Mod}_{\hf}$ whose algebras
are $N_{\infty}$-algebras over $\hf$. Then an object $X \in N_{\infty}\text{-Alg}$ is:
	\begin{itemize}
	\item a $C$-comodule with a compatible $\mathbb{P}^{\F}$-module structure,
	\item a $\mathbb{P}^{\F}$-module with a compatible $C$-comodule structure.
	\end{itemize}
To formulate the notion of \emph{compatible} one must write down a certain distributor,
in the sense of Beck, which gives a relationship between the monads and comonads
derived from the square above. 

In our case, this distributor is essentially given by either of the following pieces of data:
	\begin{itemize}
	\item the $\arep$-comodule structure on $\rpis\textup{Ops}_{\hf}$,
	\item the action of power operations on $\arep$.
	\end{itemize}
In particular, the abstract considerations above lead us to suspect that these two pieces of
information determine one another. We have already computed the $\arep$-coaction
on $\rpis\textup{Ops}_{\hf}$, and in this section we carry out the deduction of the
action of power operations on $\arep$. We first establish an abstract characterization
of the action in the form of Nishida relations, along the lines of \cite{BJ} and \cite{Ba},
and then we concretize this characterization as a computation. The main
result is Theorem \ref{action}.
\subsection{Nishida relations}

\begin{lemma}\label{lem-abstract-nishida}
 Let $E$ be a $C_2$-spectrum with
 an equivariant symmetric multiplication, and let $e \in \rpis \textup{Ops}_{\hf}$
be arbitrary. Write $\psi_R(e) = \sum e_i \otimes a_i \in \rpis \textup{Ops}_{\hf} \,\widehat{\otimes}\,
\arep$. Then, for any $x \in \rpis E$,
	\[
	\psi_R(Q^ex) = \sum (1\otimes a_i) Q^{e_i}(\psi_R(x)) \in \hf_{\star}E \otimes \arep.
	\]
\end{lemma}
\begin{proof} Represent $x$ by a map $S^V \longrightarrow E$. This yields compatible maps of spectra
	\[
	\Sigma^{V+W}\mathbb{P}_2(S^{-W}) \longrightarrow E,
	\]
and hence a map
	\[
	\rpis \Sigma^V\textup{Ops}_{\hf} \longrightarrow \hf_{\star}E
	\]
which respects the completed, right comodule structure over the equivariant dual Steenrod algebra. By
definition of the power operations, an element $\Sigma^Ve \in \rpis \Sigma^V\textup{Ops}_{\hf}$ maps
to $Q^ex$ in the target. The formula is now just a restatement of the preceding two observations. 
\end{proof}

The following reformulation of the Nishida relations in the
classical case is due to Bisson-Joyal \cite{BJ} (for $p=2$) and Baker \cite{Ba} (for $p>2$). Recall that we had
defined:
	\[
\overline{\xi}(t) = \sum_{i\ge 0} \overline{\xi}_it^{2^i}, \quad
	\overline{\tau}(t) = \sum_{i\ge 0}\overline{\tau}_it^{2^i}.
	\]

\begin{theorem}[Co-Nishida relations] Let $E$ be
a $C_2$-spectrum with
 an equivariant symmetric multiplication and let $x \in \hf_{\star}E$, then, for any $s \in \mathbb{Z}$,
we have an identity of formal power series:
	\[
	\sum_{r \in \mathbb{Z}} \psi_R\left(Q^{r\rho}x\right)t^r =
	\sum_{r \in \mathbb{Z}} Q^{r\rho}(\psi_R(x)) \overline{\xi}(t)^r + 
	\sum_{r \in \mathbb{Z}} Q^{r\rho+\sigma}(\psi_R(x)) \overline{\xi}(t)^r\overline{\tau}(t).
	\]
\end{theorem}
\begin{proof} Combine Theorem \ref{thm-comodule-structure} with Lemma
\ref{lem-abstract-nishida}, and recall that switching from left to right comodule structures
introduces a conjugation.
\end{proof}

\subsection{Computing the action}

Recall that $\rpis F(B\mu_{2+}, \hf)$ is generated by $c \in \hf^{\sigma}B\mu_2$ and $b \in \hf^{\rho}B\mu_2$.

\begin{lemma}\label{lem-ops-on-bmu}
 The action of power operations on $\rpis F(B\mu_{2+}, \hf)$ is determined by the unstable
condition, the Cartan formula, vanishing, and the following:
	\[
	\sum (Q^{r\rho}b )t^r = b + b^2 t^{-1},
	\]
	\[
	Q^{-1}c = b.
	\]
In particular:
	\[
	Q^{i\rho}(b^{2^k}) =
	\begin{cases}
	b^{2^k} & i = 0\\
	b^{2^{k+1}} & i = -2^k\\
	0 & i \ne 0, -2^k
	\end{cases}
	\]
\end{lemma}
\begin{theorem}\label{action} We have the following formulas describing the action of power operations on
the equivariant dual Steenrod algebra.
\begin{align*}
	Q^{2^k\rho}\tau_k = \tau_{k+1} + \tau_0\xi_{k+1},
	&& Q^{2^k\rho}\overline{\tau}_k = \overline{\tau}_{k+1},
	 \\
	 \beta Q^{2^k \rho}\tau_k = \xi_{k+1},
	 &&\beta Q^{2^k \rho}\overline{\tau}_k = \overline{\xi}_{k+1}\\
	Q^{2^k\rho}\xi_k = \xi_{k+1} + \xi_1\xi_k^{2} &&
	Q^{2^k\rho}\overline{\xi}_k = \overline{\xi}_{k+1}
\end{align*}
\end{theorem}

Before turning to the proof, we record a corollary.

\begin{corollary} The equivariant dual Steenrod algebra is generated by $\tau_0$ as an algebra
over the ring of power operations. The polynomial subalgebra $P_{\star} = \hf_{\star}[\xi_1, \xi_2, ...]$
is generated by $\xi_1$ over the ring of power operations. 
\end{corollary}

One can use this corollary to deduce splitting results analogous to those of Steinberger
\cite[III.2]{Hinfty}. We
will not explore this here, though our main application of the theorem is a homological splitting
result in the course of our cellular construction of $\bpr$ below.

\begin{proof}[Proof of the theorem] First we show how that the right-hand formulas follow
from the left-hand formulas. Indeed, $\overline{\tau}_k$ is defined inductively by
	\[
	\overline{\tau}_k = \tau_k + \sum_{0\le j<k} \xi^{2^j}_{k-j} \overline{\tau}_j
	\]
We want to apply $Q^{2^k\rho}$ to both sides, and
assume by induction that $Q^{2^k\rho}\overline{\tau}_j = \overline{\tau}_{j+1}$ for
$j<k$. Observe that, by the Cartan formula and degree
considerations,
	\begin{align*}
	Q^{2^k\rho}(\xi^{2^j}_{k-j}\overline{\tau}_j) &= 
	Q^{(2^k - 2^j)\rho}(\xi^{2^j}_{k-j})Q^{2^j\rho}\overline{\tau}_j\\
	&= \xi^{2^{j+1}}_{k-j}Q^{2^j\rho}\overline{\tau}_j & \text{(Squaring)}\\
	& =  \xi^{2^{j+1}}_{k-j}\overline{\tau}_{j+1} & \text{(Induction)}.
	\end{align*}
Thus,
	\begin{align*}
	Q^{2^k\rho}\overline{\tau}_k &= Q^{2^k\rho}\tau_k + \sum_{j<k} \xi^{2^{j+1}}_{k-j}\overline{\tau}_{j+1}\\
	&= Q^{2^k\rho}\tau_k + \xi_k^{2}\overline{\tau}_1 + \cdots\\
	&= \tau_{k+1} + \xi_{k+1}\tau_0 + \xi_k^{2}\overline{\tau}_1 + \cdots \\
	&= \tau_{k+1} + \sum_{0\le j<k+1} \xi^{2^j}_{k+1-j} \overline{\tau}_j\\
	&= \overline{\tau}_{k+1}.
	\end{align*}
Here we've used that $\overline{\tau}_0 = \tau_0$. Applying the Bockstein to the definition
of $\overline{\tau}_k$ and using induction and the identity $\beta \tau_j = \xi_{j}$ yields the
next formula on the right-hand column. The last formula is proved just like the first one, except we need
to know that
	\[
	Q^{(2^k - 2^j +1)\rho}\xi_{k-j}^{2^j} = 0,
	\]
which follows from the Cartan formula and vanishing of some binomial coefficients mod 2. 

Now we turn to the proof of the identities:
\begin{align*}
	Q^{2^k\rho}\tau_k = \tau_{k+1} + \tau_0\xi_{k+1},
	 \\
	 \beta Q^{2^k \rho}\tau_k = \xi_{k+1},\\
	Q^{2^k\rho}\xi_k = \xi_{k+1} + \xi_1\xi_k^{2}.
\end{align*}
The second follows immediately from the first, so we will only prove the first and last identities. We
will apply the co-Nishida relations to the $N_{\infty}$-algebra $F(B\mu_{2+}, \hf)$. For the element
$c \in \rpi_{-\sigma}F(\mu_{2+}, \hf)$ the co-Nishida relations read
	\[
	\psi_R(c) = \psi_R(Q^0c) = \left[ \sum_{r \in \mathbb{Z}} Q^{r\rho}(\psi_R(c)) \overline{\xi}(t)^r
	+ \sum_{r \in \mathbb{Z}}Q^{r\rho+\sigma}(\psi_R(c)) \overline{\xi}(t)^{r} \overline{\tau}(t)\right]_{t^0}
	\]
We are free to ignore terms with $r\ge 0$ since these do not contribute to the constant term. Also,
each term in $\psi_R(c)$ has degree $-\sigma$, so, by vanishing, we need only consider
the operations $Q^0$ and $\beta = Q^{-1} = Q^{-\rho + \sigma}$. So we can rewrite the expression as:
	\[
	\psi_R(c) = 
	\left[ Q^{0}(\psi_R(c))
	+ \beta(\psi_R(c)) \overline{\xi}(t)^{-1} \overline{\tau}(t)\right]_{t^0}.
	\]
Now recall that $\psi_R(c) = c \otimes 1 + \sum b^{2^k} \otimes \tau_k$, so we can expand out the right
hand side:
	\begin{align*}
	Q^{0}(\psi_R(c))
	+ \beta(\psi_R(c)) \overline{\xi}(t)^{-1} \overline{\tau}(t)
	&= c \otimes 1 +
	\sum_{k\ge 0}\sum_{i+j=0} Q^{i\rho}(b^{2^k}) \otimes Q^{j\rho}\tau_k \\
	&+ 
	\sum_{k\ge 0}\beta(b^{2^k}) \otimes \tau_k\overline{\xi}(t)^{-1}\overline{\tau}(t)
	+\sum_{k\ge 0}b^{2^k} \otimes \beta\tau_k\overline{\xi}(t)^{-1}\overline{\tau}(t)
	\end{align*}
Since $\beta c = b$, $\beta b^{2^k} = 0$. Also, the constant term in $\overline{\xi}(t)^{-1}\overline{\tau}(t)$
is $\overline{\tau}_0 = \tau_0$. Indeed:
	\begin{align*}
	\frac{\overline{\tau}(t)}{\overline{\xi}(t)} &=
	\frac{\overline{\tau}(t)}{t + \overline{\xi}_1 t^2 + \overline{\xi}_2 t^4 + \cdots}
	\\
	&=
	\frac{1}{t}\frac{\overline{\tau}(t)}{1 + \overline{\xi}_1 t + \overline{\xi}_2 t^3 + \cdots}\\
	&=\frac{1}{t}\left[(\overline{\tau}(t))\right] \cdot
	\left[ 1 + (\overline{\xi}_1t + \cdots) + ( \overline{\xi}_1 t + \cdots)^2 + \cdots\right]\\
	&= \frac{1}{t} (\overline{\tau}_0t + \overline{\tau}_1t^2 + \cdots)(1 + \overline{\xi}_1t + \cdots)\\
	&= \overline{\tau}_0 + (\overline{\tau}_0\overline{\xi}_1 + \overline{\tau}_1)t + \cdots
	\end{align*}
So, at this point we get the identity
	\[
	\psi_R(c) = c \otimes 1 + \sum_{k\ge 0}\sum_{i+j = 0} Q^{i\rho}(b^{2^k}) \otimes 
	Q^{j\rho} \tau_k
	+ \sum_{k\ge 0} b^{2^k} \otimes \xi_k\tau_0.
	\]
Now we expand the left hand side and simplify the
right hand side using Lemma \ref{lem-ops-on-bmu} to get
	\[
	c \otimes 1 + \sum b^{2^k} \otimes \tau_k =
	c \otimes 1 + \sum_{k\ge 0}b^{2^{k+1}} \otimes Q^{2^k\rho} \tau_k
	+ \sum_{k\ge 0} b^{2^k} \otimes \xi_k\tau_0
	\]
Comparing coefficients of $b^{2^{k+1}}$ yields the desired result:
	\[
	Q^{2^k\rho}\tau_k = \tau_{k+1} + \tau_0\xi_{k+1}.
	\]
The second identity follows from the fact that $\beta \tau_i = \xi_i$ and $\beta\xi_j = 0$, and
third identity is proved just like the first. 
\end{proof}

\section{A cellular construction of $\bpr$}

To motivate the maneuvers in the
next section, we recall Priddy's approach for constructing $\textup{BP}$. 
First, he builds a spectrum $X$ by attaching even cells to $S^0_{(2)}$ to kill odd homotopy.
There is a canonical map $X \to \textup{H}\F_2$ classifying the unit in cohomology, and
Priddy wants to show that
	\[
	\textup{H}\F_{2*}X \to \mathcal{A}_*
	\]
is surjective. In order to do so, he notes that $\zeta_1$ is automatically hit because the unit map
$S^0 \to X$ extends over the cofiber on $\eta$, which is detected by $\zeta_1$. In order to
hit the remaining $\zeta_i$, he purports to construct a map
	\[
	\mathbb{P}_2(X) \to X
	\]
by obstruction theory. The obstructions live in $H^{n+1}(\mathbb{P}_2(X); \pi_{n}X)$. Half
of these obstructions vanish by construction because $\pi_{\text{odd}}X = 0$. The other
obstructions live in \emph{odd} cohomology of $\mathbb{P}_2(X)$ with coefficients in
$\pi_*(X)$. If we knew that $\pi_*(X)$ was torsion-free, we would be done. The author
does not know how to prove that $\pi_*(X)$ is torsion-free without assuming the result
we are trying to prove, or assuming the existence of $\textup{BP}$, so it seems we have encountered
a problem in our obstruction theory.

In the equivariant case, things are even worse. Let $\ev$ denote the $C_2$-spectrum obtained
from $S^0_{(2)}$ by inductively killing the Mackey functors $\rpi_{*\rho -1}$. We might hope
to construct a map
	\[
	\mathbb{P}_2(\ev) \to \ev
	\]
by inducting up the slice tower of $\ev$. We find that the obstructions live in
	\[
	[\mathbb{P}_2(\ev), \Sigma P^{2n}_{2n}\ev]^{C_2}
	\]
where $P^{2n}_{2n}\ev$ denotes the $2n$-slice of $\ev$. Just as in the non-equivariant case,
we need to know something about the homotopy Mackey functor $\rpi_{n\rho}\ev$
in order to conclude that these groups vanish. This time it is not enough to know that these
Mackey functors are torsion free, we must check that their underlying groups are torsion-free
\emph{and} that they have a surjective restriction map. 

The observation that saves the argument is as follows: we do not need the map
$\mathbb{P}_2(\ev) \to \ev$ all at once. Instead, we can inductively define maps
from a sub-complex of $\mathbb{P}_2(\ev)$ into $\ev$, and use these maps
to define \emph{some} of the power operations we need. Something similar works, and more
easily, in the non-equivariant case. 

The induction requires some rather technical homological algebra, for which we give
proofs that work, but are perhaps not so elegant. We imagine that developing the theory
of the slice filtration in the context of chain complexes of $RO(C_2)$-graded Mackey functors
would result in more conceptual proofs of these results, but we do not see immediately just
how this should go.

Finally, we use this characterization of the homology of $\ev$ to determine its homotopy,
its geometric fixed points, its slice tower, and to identify if with $\bpr$. Though, once again,
we emphasize that the fact that $\ev \cong \bpr$ admits a much simpler proof but it is
not in the spirit of this paper.
\subsection{Review of the slice filtration}
We will need to build maps using obstruction theory. In the classical case, obstruction
theory comes from Postnikov towers. We will need an equivariant version of these originating
in work of Dugger and developed much further by Hill-Hopkins-Ravenel and Hill.
\begin{definition} The set of {\bf slice cells} is
	\[
	\{ C_{2+} \wedge S^m, S^{m\rho}, S^{m\rho -1} | m \in \mathbb{Z}\}.
	\]
The dimension of a slice cell is the dimension of the underlying spheres.
\end{definition}

Define a full sub-homotopy theory $\textup{Sp}^{C_2}_{>n} \subset \textup{Sp}^{C_2}$
containing slice cells of dimension $>n$ and closed under equivalence,
homotopy colimits, and extension. 

The inclusion $\textup{Sp}^{C_2}_{>n} \hookrightarrow \textup{Sp}^{C_2}$ admits
a homotopical right adjoint $P_{n+1}: \textup{Sp}^{C_2} \longrightarrow \textup{Sp}^{C_2}_{>n}$
and we define $P^{n}$ by the natural cofiber sequence
	\[
	P_{n+1} \to \textup{id} \to P^n.
	\]
The inclusions $\textup{Sp}^{C_2}_{>n} \subset \textup{Sp}^{C_2}_{>n-1}$ yield natural
maps $P^n \to P^{n-1}$ so every $E \in \textup{Sp}^{C_2}$ has an associated tower:
	\[
	\xymatrix{
	&& \vdots\ar[d] &\\
	&&P^nE \ar[d] &\ar[l]P^n_nE\\
	&&P^{n-1}E \ar[d]&\ar[l]P^{n-1}_{n-1}E\\
	&&\vdots \ar[d]&\\
	E\ar[rr]\ar[uurr]\ar[uuurr]&&P^0E&\ar[l]P^0_0E
	}
	\]
The fiber $P^n_nE$ is called the {\bf $n$-slice} of $E$, and the tower is called the {\bf slice tower} of $E$.
The natural map $E \to \holim P^nE$ is a weak equivalence and, for each $k$, the tower of
homotopy Mackey functors $\{\rpi_kP^nE\}$ is pro-isomorphic to the constant tower $\{\rpi_kE\}$
\cite[Thm 4.42]{HHR}.

We will need a formula for the $n$-slices in terms of the homotopy Mackey functors of $E$, but
to describe this formula we need a definition.

\begin{definition} If $\underline{M}$ is a $C_2$ Mackey functor let $P^0\underline{M}$ denote the
Mackey functor:
	\[
	P^0\underline{M} = 
	\begin{gathered}
	\xymatrix{
	\textup{im}(\textup{res}: \underline{M}(C_2/C_2) \to \underline{M}(C_2))\ar@/_/[d]\\
	\underline{M}(C_2) \ar@/_/[u]
	}
	\end{gathered}
	\]
\end{definition}

\begin{theorem}[HHR, Hill] \label{slice-tower} For any $C_2$-spectrum $E$, the slices are given by
	\[
	P^{2n}_{2n}E = \Sigma^{n\rho}\textup{H}P^0\rpi_{n\rho}E,
	\]
	\[
	P^{2n-1}_{2n-1}E = \Sigma^{n\rho -1}\textup{H}\rpi_{n\rho -1}E.
	\]
\end{theorem}
\subsection{Construction of $\textup{Even}(S^0)$}
{\it The remainder of \S6 is implicitly localized at 2}.
\newline
Suppose given a $C_2$-spectrum $Y$. We will say that $Z$ is obtained from $Y$ by
{\bf non-trivially attaching $n\rho$-cells} if there is a cofiber sequence
	\[
	\bigvee W_{\alpha} \stackrel{f}{\longrightarrow} Y \longrightarrow Z
	\]
where each $W_{\alpha}$ is a slice cell of dimension $2n-1$ and $\textup{ker}(f_*) \subset ([C_2], 2) \cdot \bigoplus_{\alpha}
\rpi_{n\rho-1}(W_{\alpha})$. 

\begin{construction} Let $\textup{Even}(S^0)_0 = S^0$. Given $\textup{Even}(S^0)_{k}$, non-trivially attach cells of dimension
$2k+2$ to obtain $\textup{Even}(S^0)_{k+1}$ with $\rpi_{(k+1)\rho -1}\textup{Even}(S^0)_{k+1} = 0$. Let $\textup{Even}(S^0) = \textup{colim }\textup{Even}(S^0)_k$.
Note that $\textup{Even}(S^0)$ comes equipped with a map $S^0 \longrightarrow \textup{Even}(S^0)$ which 
we'll call the unit.
\end{construction}

An immediate consequence of the construction is a splitting result in homology.

\begin{lemma}\label{even-split} $\hz \wedge \textup{Even}(S^0) \cong \hz \wedge W$ as $\hz$-modules, where $W$
is a wedge of even-dimensional slice cells.
\end{lemma}
\begin{proof} Consider the cofiber sequence:
\[
	\xymatrix{
	\bigvee W_{\alpha} \ar[r] \ar[dr]&
	\textup{Even}(S^0)_k\ar[r]\ar[d] &
	\textup{Even}(S^0)_{k+1} \ar[r] &
	\bigvee \Sigma W_{\alpha}\\
	&\hz \wedge \textup{Even}(S^0)_k&&
	}
	\]
By induction we may assume that $\hz \wedge \textup{Even}(S^0)_k \cong \hz \wedge Y$
where $Y$ is a wedge of even-dimensional slice cells. Thus, the diagonal map is a sum of elements in $\rpi_{m\rho -1}\hz$ for various $m$. But these
Mackey functors are all zero, so the map is null and we get the desired splitting.

\end{proof}

A slightly less immediate consequence of the construction is the vanishing of the mod 2 Hurewicz map.

\begin{lemma}\label{hurewicz-vanish}
The Hurewicz map $\rpi_{k\rho} \ev \to \hf_{k\rho}\ev$ is zero for $k>0$.
\end{lemma}
\begin{proof} We argue as in Priddy.
We have a commutative diagram:
	\[
	\xymatrix{
	\rpi_{n\rho}\textup{Even}(S^0)_n \ar[r] \ar[d]& \bigoplus \rpi_{n\rho}\Sigma W_{\alpha} \ar[r]^{\partial} \ar[d]
	& \rpi_{n\rho}\Sigma \textup{Even}(S^0)_{n-1}\\
	\hf_{n\rho}\textup{Even}(S^0)_n \ar[r] & \bigoplus \hf_{n\rho}\Sigma W_{\alpha} &
	}
	\]
The bottom horizontal arrow is an isomorphism for $n>0$ because $\hf_{n\rho}\textup{Even}(S^0)_{n-1} = 0$ by induction on $n$.
Thus $j(g)$ is in the kernel of $\partial$ but does not lie in 
$([C_2], 2)\cdot \bigoplus \rpi_{n\rho}\Sigma W_{\alpha}$, otherwise it would vanish under the Hurewicz map.

For obstruction theory purposes, we will need to know a bit about the homotopy groups
of $\textup{Even}(S^0)$.
\end{proof}

\begin{lemma}\label{odd-zero} $\rpi_{k\rho -1}\textup{Even}(S^0) = 0$ for all $k\ge 0$.
\end{lemma}
\begin{proof} Consider the cofiber sequence
	\[
	\bigvee W_{\alpha} \longrightarrow \textup{Even}(S^0)_n \longrightarrow \textup{Even}(S^0)_{n+1}
	\]
where each $W_{\alpha}$ is a slice cell of dimension $2n-1$. This yields an exact sequence:
	\[
	\rpi_{k\rho-1}\textup{Even}(S^0)_n \longrightarrow \rpi_{k\rho -1}\textup{Even}(S^0)_{n+1} \longrightarrow \bigoplus \pi_{k\rho - 2} 
	W_{\alpha}.
	\]
Recall that
	\[
	\rpi_{k\rho -2}S^{n\rho-1} = \rpi_0S^{(n-k)\rho + 1} = 0, \quad n\ge k. 
	\]
It follows that
	\[
	\rpi_{k\rho -1}\textup{Even}(S^0) = \rpi_{k\rho -1}\textup{Even}(S^0)_{k} = 0.
	\]
\end{proof}

We can also compute the first nonvanishing homotopy group.

\begin{lemma} $\rpi_0\textup{Even}(S^0) = \underline{\mathbb{Z}}_{(2)}$. 
\end{lemma}
\begin{proof} The same connectivity argument 
as above shows that $\rpi_0\textup{Even}(S^0) = \rpi_0\textup{Even}(S^0)_1$. The spectrum $\textup{Even}(S^0)_1$ is built
from $S^0$ by killing the elements in $\rpi_{\rho -1}S^0$. This group is generated by the Hopf map
$\eta: S^{\sigma} \to S^0$ so $\textup{Even}(S^0)_1 = \textup{cofib}(\eta)$. The attaching map then gives an exact sequence:
	\[
	\xymatrix{
	\rpi_0(S^{\sigma}) \ar[r] & \rpi_0(S^0) \ar[r]& \rpi_0(\textup{Even}(S^0)_1) \ar[r] & \rpi_{-1}(S^{\sigma})\\
	& \underline{A} \ar[u]_-{\cong} \ar[r]& \rpi_0(\textup{Even}(S^0)_1)\ar@{=}[u] \ar[r]& 0 \ar@{=}[u]
	}
	\]
where $\underline{A}$ denotes the (2-local) Burnside Mackey functor. The image of $\eta_*$ is precisely
the ideal $([C_2]-2)$, so $\rpi_0(\textup{Even}(S^0)_1) = \underline{A}/([C_2]-2) \cong \underline{\mathbb{Z}}_{(2)}$,
which was to be shown. 
\end{proof}

The vanishing in (\ref{odd-zero}) is enough to build a ring structure on $\textup{Even}(S^0)$.

\begin{lemma} The spectrum $\ev$ admits a multiplication $\ev \wedge \ev \to \ev$
compatible with the unit $S^0 \to \ev$.
\end{lemma}
\begin{proof} We build the map by induction up the slice tower of $\ev$.  Since the odd
slices of $\ev$ are zero, we are only concerned with obstructions in
	\[
	[\ev \wedge \ev, \Sigma^{n\rho+1}\textup{H}P^0\rpi_{n\rho}\ev]^{C_2}.
	\]
Every Mackey functor with injective restriction maps is a module over 
$\underline{\mathbb{Z}}$. So
we need only show that
	\[
	\fpi_{k\rho-1}\textup{Map}_{\hz}(\hz \wedge (\ev \wedge \ev), \textup{H}\underline{M}) = 0.
	\]
where $\underline{M}$ has injective restriction maps.

By (\ref{even-split}), we know
that $\hz \wedge \ev \wedge \ev$ splits as a wedge of even-dimensional slice cells. So it's enough to
check that $\fpi_{k\rho -1}\textup{H}\underline{M} = 0$.
That's essentially true by the definition of the slice filtration. We can verify it just to be sure, though.
It amounts to showing that $\fpi_{k\rho -1}\textup{H}\underline{M} = 0$. By connectivity, we need only
check that case when $k=1$, so we want to know $\fpi_{\sigma}\textup{H}\underline{M}$. From
the cofiber sequence
	\[
	C_{2+} \to S^0 \to S^\sigma
	\]
we see that $\fpi_{\sigma}\textup{H}\underline{M} = \textup{ker}(\textup{res})$. This vanishes by assumption on
$\underline{M}$.
\end{proof}

\subsection{Some algebraic preliminaries on $P_{\star}$}
Let $P_{\star} \subset \mathcal{A}_{\star}$ denote the left $\hf_{\star}$-submodule
$\hf_{\star}[\xi_1, \xi_2, ...]$. This is a left $\mathcal{A}_{\star}$-comodule algebra, but it is
\emph{not} a Hopf algebroid. The trouble is that
	\[
	\eta_R(u_{\sigma}) = a_{\sigma}\tau_0 + u_{\sigma},
	\]
so $P_{\star}$ is not even a right $\hf_{\star}$-module. If we want a Hopf algebroid, we can take
$(\underline{\F}, P^{\circ}_{\star})$ defined by
$P^{\circ}_{\star} = \underline{\F}[\xi_1, \xi_2, ...]$.

There is an extension of Hopf algebroids
	\[
	(\underline{\F}[a_{\sigma}], P^{\circ}_{\star}[a_{\sigma}]) 
	\to (\hf_{\star}, \mathcal{A}_{\star}) \to (\hf_{\star}, \Lambda)
	\]
where
	\[
	\Lambda = \hf_{\star}[\tau_0, \tau_1, ...]/ (\tau_i^2 = a_{\sigma}\tau_{i+1}).
	\]
\begin{lemma} As a left $\mathcal{A}_{\star}$-comodule algebra,
	\[
	P_{\star} \cong \mathcal{A}_{\star} \Box_{\Lambda} \hf_{\star}.
	\]
\end{lemma}
\begin{proof} Since $P_{\star} = \hf_{\star}\otimes_{\uF[a_{\sigma}]}P^{\circ}_{\star}[a_{\sigma}]$
the result follows from
\cite[A1.1.16]{Rav2}.
\end{proof}

\begin{corollary}\label{change-of-rings} $\underline{\textup{Ext}}_{\arep}(P_{\star}) \cong 
\underline{\textup{Ext}}_{\Lambda}(\hf_{\star})$.
\end{corollary}

We'll need to compute something about these Ext groups for the next section. 

\begin{proposition}\label{ext-computation} 
	\begin{enumerate}[(i)]
	\item The Mackey functors $\underline{\textup{Ext}}^{s, s+n\rho-1}_{\mathcal{A}_{\star}}(P_{\star})$
	vanish,
	\item the Mackey functors
	$\underline{\textup{Ext}}^{s, s+n\rho+1}_{\mathcal{A}_{\star}}(P_{\star})$
	vanish when evaluated at $C_2$,
	\item if we let $v_i:= [\tau_i]$ in the cobar complex, then
		\[
		\bigoplus \underline{\textup{Ext}}^{s, s+n\rho}_{\mathcal{A}_{\star}}(P_{\star})
		\cong \uF[v_0, v_1, ...].
		\]
	\end{enumerate}
\end{proposition}
\begin{proof} First filter $\Lambda$ by declaring that $\tau_i$ have filtration $2^{i+1}-2$ and
extending multiplicatively. The associated graded is the Hopf algebroid 
$(\hf_{\star}, \hf_{\star}[\tau_i]/ \tau_i^2)$.
Write $E(\tau_i)$ for the exterior algebra on the primitive $\tau_i$ over $\uF$.
Then examining the cobar complex
reveals
	\[
	\mathrm{Cotor}_{\hf_{\star}[\tau_0, \tau_1, ...]/(\tau_i^2)}(\hf_{\star}, \hf_{\star})
	= \mathrm{Cotor}_{E(\tau_0, \tau_1, ...)}(\hf_{\star}, \uF),
	\]
where $\hf$ is viewed as a right $E(\tau_0, \tau_1, ...)$ comodule via the right unit
	\[
	\eta_R: \hf_{\star} \longrightarrow \hf_{\star}\otimes_{\uF}E(\tau_0, \tau_1, ...).
	\]
The right unit on $\hf_{\star}$ only involves $\tau_0$, so we get an isomorphism
	\[
	\mathrm{Cotor}_{E(\tau_0, \tau_1, ...)}(\hf_{\star}, \uF)
	\cong \mathrm{Cotor}_{E(\tau_0)}(\hf_{\star}, \uF)[v_1, v_2, ...]
	\]	
where $v_i$ corresponds to $[\tau_i]$ in the cobar complex. To compute the remaining cotor,
we use the standard resolution by relative injectives:
	\[
	\uF \to E(\tau_0) \to \Sigma E(\tau_0) \to \Sigma^2E(\tau_0) \to \cdots
	\]
and apply $\hf_{\star}\Box_{E(\tau_0)}(-)$ to get the complex:
	\[
	\hf_{\star} \to \Sigma \hf_{\star} \to \Sigma^2\hf_{\star} \to \cdots.
	\]
The differential is non-trivial, but can be computed using the formula for the right unit. The answer,
additively, is:
	\[
	\mathrm{Cotor}_{E(\tau_0)}(\hf_{\star}, \uF) = 
	\left(\uF[u_{\sigma}^2, a_{\sigma}] \oplus J\right)[v_0]/ (v_0a_{\sigma}, v_0J).
	\]
Here $J$ denotes the square-zero piece:
	\[
	J = \uF\left\{ \dfrac{\theta}{a_{\sigma}^ku_{\sigma}^{2n+1}} \right\}. 
	\]
The proposition now follows by noting that there is no room for
differentials near the
degrees of interest.
\end{proof}

In the next section, we will be concerned with the $\arep$-comodule structure of $\hf_{\star}\ev$.
For the sake of inductive arguments, it will be helpful to record two lemmas which allow one
to conclude information based off of the behavior of a comodule in a range of dimensions.
The first is trivial, and the second is less so.

\begin{lemma}[Approximate Milnor-Moore] \label{approx-split} Let $(B, \Sigma)$ be a graded connected
Hopf algebroid, $M$ a graded connected right $\Sigma$-comodule algebra,
and $C = M \Box_{\Sigma} B$. Suppose
	\begin{enumerate}[(i)]
	\item there is a comodule algebra map $f: M \to \Sigma$ surjective in degrees $\le N$,
	\item $C$ is a $B$-module and as such is a direct summand of $M$ in degrees $\le N$.
	\end{enumerate}
Then $M \cong C \otimes_B\Sigma$ as a left $C$-module and right $\Sigma$-comodule
in degrees $\le N$. 
\end{lemma}
\begin{proof} The proof is the same as in A.1.1.17 of \cite{Rav2}.
\end{proof}

\begin{definition} We will say that an $RO(C_2)$-graded
$\hf_{\star}$-module $M$ is {\bf pure} if it is of
the form
	\[
	M \cong \bigoplus \hf_{\star}\{e_{\alpha}\} \oplus \bigoplus \left[\hf_{\star}\right]_{C_2}\{f_{\beta}\}
	\]
where $|e_{\alpha}| = n_{\alpha}\rho$ and $|f_{\beta}| = 2m_{\beta}$ for some integers
$n_{\alpha}$ and $m_{\beta}$ and some indexing set. If moreover, we can choose the $e_{\alpha}$
and $f_{\beta}$ such that $n_{\alpha}, m_{\beta}\ge 0$ for all $\alpha, \beta$ then we will
say that $M$ is {\bf pure and connected}. 
\end{definition}

\begin{lemma}[Approximate Ext] \label{approx-ext}
Let $M$ and $M'$ be $\arep$-comodules which are pure and connected 
as $\hf_{\star}$-modules,
and fix a positive integer $N$. Suppose that
there is an
isomorphism on the submodules
	\[
	\langle e_{\alpha}, f_{\beta} | n_\alpha, m_{\beta}\le N\rangle \cong
	\langle e'_{\alpha}, f'_{\beta} | n'_{\alpha}, m'_{\beta}\le N\rangle
	\]
for some decomposition as in the definition above. Suppose, moreover, that
this is an isomorphism of comodules modulo the generators not listed. Then the Mackey functors
	\[
	\underline{\textup{Ext}}^{s, V}_{\arep}(M) \quad \underline{\textup{Ext}}^{s, V}_{\arep}(M')
	\]
are 
	\[
	\begin{cases}
	\textup{isomorphic }& V-s = k\rho-1, \,\,k\le N\\
	\textup{isomorphic mod }a_{\sigma} & V-s = k\rho, \,\,k\le N\\
	\textup{isomorphic on underlying groups } & V-s = k\rho, k\rho +1\textup{ and }k\le N
	\end{cases}
	\]
\end{lemma}
\begin{proof} We will show that the terms of the reduced
cobar complex agree in the appropriate
dimensions, and then make sure that the cycles and boundaries remain the same. Recall
that
	\[
	\textup{Cobar}^s(M) := \rarep^{\otimes s} \otimes_{\hf_{\star}} M.
	\]
As an $\mathbb{F}_2$-vector space, this is generated by elements of the form
	\[
	c [a_1 | a_2 | \cdots | a_s] e\quad \text{and} \quad c [ a_1 | a_2 | \cdots | a_s] \frac{f}{u^r_{\sigma}}
	\]
where $c \in \hf_{\star}$, the $a_i$ are non-constant monomials in $\tau_i$ and $\xi_i$,
$e = e_{\alpha}$ for some $\alpha$ and $f = f_{\beta}$ for some $\beta$. 

\underline{Step 1}. The degree $s+k\rho$ part of $\textup{Cobar}^s(M)$ is contained
in the span of $e_{\alpha}$ and $f_{\beta}$ with $n_{\alpha}, m_{\beta} \le k$. 

If $x$ is an
element in an $RO(C_2)$-graded module, we will write $\textup{deg}(x) = x^0 + x^1\sigma$. Then
	\[
	\textup{deg}(c [a_1 | a_2 | \cdots | a_s] e) = (c^0 + \sum a_i^0 + e^0) +
	(c^1 + \sum a_i^1 + e^1)\sigma,
	\]
	\[
\textup{deg}\left(c [a_1 | a_2 | \cdots | a_s] \frac{f}{u^r_\sigma}\right) = (c^0 + \sum a_i^0 + f^0 -r) +
	(c^1 + \sum a_i^1 + f^1 + r)\sigma.
	\]
Observe that $\sum a_i^0\ge s$ since the minimal $a^0_i$ that can occur for
a monomial $a_i$ in $\rarep$ is $1$. Also observe that $\sum a_i^1 \ge 0$ since every non-constant
monomial has this property. If we restrict to those basis elements in degree $s + k\rho$, we see
	\[
	\begin{cases}
	e^0 \le k - c^0\\
	e^1 \le k - c^1
	\end{cases}
	\]
	\[
	\begin{cases}
	f^0 \le k - c^0 + r\\
	f^1 \le k - c^1 - r
	\end{cases}
	\]
Now we use the fact that $e = e_{\alpha}$ and $f = f_{\beta}$ so that $e^0=e^1 = n$ for some
$n\ge 0$ and $f^1 = 0$ and $f^0 = 2m$ for some $m\ge 0$. Then these inequalities become:
\[
	\begin{cases}
	n \le k - c^0\\
	n \le k - c^1
	\end{cases}
	\]
	\[
	\begin{cases}
	2m \le k - c^0 + r\\
	0 \le k - c^1 - r
	\end{cases}
	\]
But $c \in \hf_{\star}$ is only nonzero when $c^0$ and $c^1$ have opposite sign, so at least
one of these is nonnegative and we conclude that $n\le k$. If $c^0 <0$ then $c$ becomes
$a_{\sigma}$-divisible, while $[\hf_{\star}]_{C_2}$ is $a_{\sigma}$-torsion and we conclude
that our generator vanishes in the second case. If $c^0\ge 0$ and divisible by $a_{\sigma}$
we're also okay. If $c^0\ge 0$ and not divisible by $a_{\sigma}$, then it must be of the
form $u_{\sigma}^{\ell}$, in which case $c^0 = -c^1$. Adding the last two inequalities
then yields $2m \le 2k$ and hence $m\le k$.

\underline{Step 2}. The degree $s+k\rho -1$ part of $\textup{Cobar}^s(M)$ is contained
in the span of 

The same argument for degree $s +k\rho-1$ yields the inequalities
\[
	\begin{cases}
	n \le k - c^0-1\\
	n \le k - c^1
	\end{cases}
	\]
	\[
	\begin{cases}
	2m \le k - c^0 + r-1\\
	0 \le k - c^1 - r
	\end{cases}
	\]
If $c^0\ge 0$ then $n \le k-1$. But if $c^1\ge 0$ then in fact $c^1 \ge 2$ by the gap in the
equivariant homology of a point, so then $n\le k-2$. In either case, $n\le k-1$. The same
argument as before shows that $m\le k- \frac{1}{2}$. But $m$ is an integer, so
$m\le k- 1$. 

Combining the results of these two arguments, and the identical result
for $M'$ in place of $M$, we have shown that
	\[
	\left(\textup{Cobar}^s(M)\right)_V \cong \left(\textup{Cobar}^s(M')\right)_V
	\]
when
	\[
	V-s = \begin{cases}
	k\rho &k\le N,\\
	k\rho -1 & k\le N+1
	\end{cases}
	\]
Now we need to check that this induces an isomorphism in homology in the degrees claimed
in the statement of the lemma.
The only difficult thing to check is that the witnesses for boundaries do not have too high of a
degree. 

\underline{Step 3}. Both $M$ and $M'$ have identical
 $\underline{\textup{Ext}}^{s,V}$ when $V -s = k\rho - 1$ for $k\le N$.
 
Suppose $y \in \textup{Cobar}^{s+1}(M)$ has degree $s+1+(N\rho -1) = s+N\rho$,
and there is some $x \in \textup{Cobar}^s(M)$ with $dx = y$. Since $x$ has degree $s+N\rho$,
what we've done so far guarantees that $y$ remains a boundary in $\textup{Cobar}^{s+1}(M')$.

\underline{Step 4}. Both $M$ and $M'$ have the same $\underline{\textup{Ext}}^{s,V}$ modulo
$a_{\sigma}$, when $V-s = k\rho$ for $k\le N$.

It follows from what we've done so far that the cycles in homological degree $s$ are the same
for both $M$ and $M'$. The only thing we have to worry about are the boundaries.

Suppose $y \in \textup{Cobar}^{s+1}(M)$ has degree $s+1+N\rho$
and $x \in \textup{Cobar}^{s}(M)$ has $dx = y$. Then we claim there is some $x'$ in the
span of $\rarep^{\otimes s} \otimes \{e_{\alpha}, f_{\beta}\}$ for $n_{\alpha}, m_{\beta} \le N$
such that $dx' = y$ modulo $a_{\sigma}$.

To that end, let's consider elements in $\textup{Cobar}^{s}(M)$ of degree $s+1+N\rho$. Using
notation as before, we must have
	\[
	\begin{cases}
	n \le N - c^0+1\\
	n \le N - c^1
	\end{cases}
	\]
	\[
	\begin{cases}
	m \le N + \frac{1}{2}
	\end{cases}
	\]
Since $m$ is an integer, $m\le N$. Thus, writing $x$ in terms of generators, the only possible
components which are in high degrees are those with $e_{\alpha}$ having $n_{\alpha} = N+1$.
By the inequalities above, this implies $c^0 = 0$ and $\sum a_i^0 = s$. So we are concerned
about generators of the form
	\[
	a_{\sigma}^r[ a_1 | \cdots | a_s] e, \quad\text{where}\quad |e| = (N+1)\rho, \,\,\,
	 a_i \in \{\tau_0, \xi_1\}. 
	\]
In order for this element to have degree $s+1+N\rho$, we must have $r\ge 1$, so this
element is divisible by $a_{\sigma}$. The differentials are linear over $\uF[a_{\sigma}]$
(since $a_{\sigma}$ is primitive), so the claim follows.

\underline{Step 5}. Both $M$ and $M'$ have the same \emph{underlying} group for
$\underline{\textup{Ext}}^{s,V}$ when $V-s = k\rho+1$ and $k\le N$.

This follows from the same vanishing argument as in the previous step using
the free generators.
\end{proof}

\begin{remark} One should think that things shouldn't be so hard. The author welcomes
any suggestions.
\end{remark}

\subsection{Properties of $\textup{Even}(S^0)$ and equivalence with $\bpr$}

\begin{lemma}\label{coaction-factors} There is a factorization:
	\[
	\xymatrix{
	\uF\{e^i_{\alpha}, f^j_{\beta}|i,j\le k\} \ar@{-->}[r]\ar[d] &
	 P^{\circ}_{\star} \otimes_{\uF} \uF\{e^i_{\alpha}, f^j_{\beta}|i,j\le k\}\ar[d]\\
	 \hf_{\star}\ev_k \ar[r]_-{\psi} & \arep \otimes_{\hf_{\star}} \hf_{\star}\ev
	}
	\]
This implies that $\uF\{e_{\alpha}^i, f^j_{\beta}| i,j\le k\}$ is a $P^{\circ}_{\star}$-comodule
and that the full coaction factors as:
	\[
	\hf_{\star} \ev_k \to P_{\star} \otimes_{\hf_{\star}} \hf_{\star}\ev_k \to \arep \otimes_{\hf_{\star}}
	\hf_{\star}\ev.
	\]
\end{lemma}
\begin{proof} This follows formally for degree reasons.
\end{proof}

We
would like to show that
	\[
	\hf_{\star}\ev = P_{\star}.
	\]
We will prove this by a simultaneous induction on the following properties:
	\begin{enumerate}[(i)]
	\item $\rpi_{*\rho}\ev$ is a constant Mackey functor in a range,
	\item $\ev$ admits a partially defined map $\mathbb{P}_2(\ev) \to \ev$,
	\item the canonical map $\hf_{\star}\ev \to P_{\star}$ hits all the $\overline{\xi}_k$ in a range.
	\end{enumerate}
Property (i) will allow us, by obstruction theory, to extend our symmetric multiplication in (ii) a bit further.
This will allow us to define a certain power operation on the top available $\overline{\xi}_k$ to get the next one,
which extends the range in (iii). This gives new input for the Adams spectral sequence, and allows us
to extend the range in (i) and repeat the process.

The next two lemmas will be used to show the vanishing of certain obstructions.

\begin{lemma}\label{coh-vanish-torsion-free}
Suppose $X$ has the property that
	\[
	\hz \wedge X \cong \hz \wedge W
	\]
where $W$ is a wedge of even-dimensional slice cells. Let $\underline{M}$ be a
torsion-free, constant Mackey functor. Then, for all $n \in \mathbb{Z}$,
	\[
	\textup{H}\underline{M}^{k\rho +1}(\mathbb{P}_2(X)) = 0.
	\]
\end{lemma}
\begin{proof} We are immediately reduced to the
case where $\underline{M} = \underline{\mathbb{Z}}$. Notice that
	\[
	\hz^{\star}(\mathbb{P}_2(X)) = \rpi_{-\star}\textup{Map}_{\hz}(\hz \wedge \mathbb{P}_2(X), \hz)
	\]
where $\textup{Map}_{\hz}(-,-)$ is the derived mapping spectrum in the homotopy theory of $\hz$-modules.
We will split the source and verify the vanishing on each piece.

Recall (\ref{lem-change-base}) that
	\[
	\hz \wedge \mathbb{P}_2(X) \cong \mathbb{P}^{\underline{\mathbb{Z}}}_2(\hz \wedge X).
	\]
By our assumption on $X$, repeated application of (\ref{lem-distr}) yields a splitting
	\[
	\hz \wedge \mathbb{P}_2(\textup{Even}(S^0))
	\cong
	\left(\hz \wedge W_1\right) \vee \left(\hz \wedge W_2\right) \vee
	 \left(\hz \wedge W_3\right) \vee \left(\hz \wedge W_4\right)
	\]
where:
	\begin{itemize}
	\item $W_1$ is a wedge of extended powers $\mathbb{P}_2(S^{m\rho})$,
	\item $W_2$ is a wedge of spheres of the form $S^{m\rho}$,
	\item $W_3$ is a wedge of extended powers $C_{2+} \wedge \mathbb{P}_2(S^{2m})$,
	\item $W_4$ is a wedge of spheres of the form $C_{2+} \wedge S^{2m}$.
	\end{itemize}
We know that $\hz^{k\rho +1}(S^{m\rho})$ and $\hz^{k\rho +1}(C_{2+} \wedge S^{2m})$ vanish.
The groups $\hz^{k\rho +1}(C_{2+} \wedge \mathbb{P}_2(S^{2m}))$ are just the underlying cohomology
groups $\textup{H}\mathbb{Z}^{2k+1}(\mathbb{P}_2(S^{2m}))$ which vanish classically (group cohomology
of $\Sigma_2$ with torsion-free coefficients vanishes in even degrees). So we are left with showing that
$\hz^{k\rho +1}(\mathbb{P}_2(S^{m\rho})) = 0$. This follows from our computation of the action
of the Bockstein on $\hf_{\star}(\mathbb{P}_2(S^{m\rho}))$ (recall that we have implicitly localized at
2).
\end{proof}

\begin{lemma}\label{easy-vanishing}
 Suppose $X$ admits a filtration with associated graded $\textup{gr}(X) = W$,
a wedge of even-dimensional slice cells of dimension $\le 2N$. Then, for any Mackey
functor $\underline{M}$,
	\[
	[F_{2j}\mathbb{P}_2(X), \Sigma^{k\rho+1}\textup{H}\underline{M}] = 0,
	\]
whenever $k\ge j+N$. 
\end{lemma}

\begin{proof} Combining the filtration of $X$ with the filtration of $\mathbb{P}_2$, we are reduced to
checking that each of the following vanish for $m \le j+N\le k$:
	\[
	[S^{m\rho}, \Sigma^{k\rho+1}\textup{H}\underline{M}], \quad
	[S^{m\rho-1}, \Sigma^{k\rho +1}\textup{H}\underline{M}],\quad
	[C_{2+} \wedge S^{2m}, \Sigma^{k\rho +1}\textup{H}\underline{M}],\quad
	[C_{2+} \wedge S^{2m-1}, \Sigma^{k\rho +1}\textup{H}\underline{M}].
	\]
But we have $m-k \le 0$ so this follows by connectivity.
\end{proof}

Finally, we are ready for the main technical theorem pinning down the homology of $\ev$.

\begin{theorem} The map $\ev \to \hf$ induces an isomorphism of $\arep$-comodules:
	\[
	\hf_{\star}\ev \stackrel{\cong}{\longrightarrow} P_{\star}.
	\]
\end{theorem}
\begin{proof} We will reduce to inductively showing that the image of 
$\hf_{\star}\ev$ contains all the $\overline{\xi}_i$. Then
we'll reduce this to an Ext calculation which we perform by comparing 
$\underline{\textup{Ext}}_{\arep}(\hf_{\star}\ev)$ to $\underline{\textup{Ext}}_{\arep}(P_{\star})$
in a range. Throughout, we will use the equivariant Adams spectral sequence for
the cohomology theory $\hf$ found in \cite[Cor. 6.47]{HK}.
\newline\newline
\underline{Step 1}. It's enough to show that the image of $\hf_{\star}\ev$ contains
$\overline{\xi}_i$ for all $i\ge 1$.

Indeed, in this case we conclude that, in the notation of Lemma \ref{coaction-factors},
the map $\uF\{e^i_{\alpha}, f^j_{\beta}\} \to P^{\circ}_{\star}$ is surjective. By the comodule
splitting lemma, we conclude that $\uF\{e^i_{\alpha}, f^j_{\beta}\}$ is cofree over
$P^{\circ}_{\star}$. But then $\hf_{\star} \otimes_{\uF}\uF\{e^i_{\alpha}, f^j_{\beta}\} = \hf_{\star}\ev$
is of the form $P_{\star} \otimes C$ as a left $\mathcal{A}_{\star}$-comodule, where $C$ is
primitive and generated in even degrees. The Adams spectral sequence then takes the form
	\[
	\underline{\textup{Ext}}_{\arep}(P_{\star}) \otimes C \Rightarrow \rpi_{\star}\ev.
	\]
Since these $\textup{Ext}$-Mackey functors vanish in degrees $V-s = n\rho -1$, every
element in homotopy dimension $n\rho$ is a permanent cycle. In particular, all of the
positive degree elements of $C$ would survive to homotopy classes which are detected
by the mod 2 Hurewicz map. There are no such elements, by Lemma \ref{hurewicz-vanish},
so $C$ is generated in degree 0. But there is only room for one such generator and we conclude
$\hf_{\star}\ev = P_{\star}$, as desired. This completes the first step.

Note that the unit map $1: S^0 \to \ev_1$ extends over $\textup{cofib}(\eta)$, by construction. 
The left comodule action on the top class $e_\rho$ is:
	\[
	e_{\rho} \mapsto 1 \otimes e_{\rho} + \overline{\xi}_1 \otimes 1.
	\] 
So we may assume, by induction on $k\ge 1$, that
	\begin{itemize}
	\item The elements $\overline{\xi}_1, ..., \overline{\xi}_k$ are
	in the image of the map $\hf_{\star}\ev_{2^k - 1} \to P_{\star}$.
	\end{itemize}
\underline{Step 2}. If $\overline{\xi}_1, ..., \overline{\xi}_k$ are in the image, then
$\overline{\xi}_{k+1}$ is in the image if we can construct a dotted arrow making the following
diagram commute:
	\[
	\xymatrix{
	F_{2^{k+1}}(\mathbb{P}_2(\ev_{2^k-1})) \ar[r]\ar[d] & \mathbb{P}_2(\hf) \ar[d]\\
	\ev \ar[r] & \hf
	}
	\]
Indeed, this subcomplex of the extended power is just large enough to contain the operation
$Q^{2^k\rho}$. Moreover, $\hf_{\star}\ev_{2^k-1}\to \hf_{\star}\ev$ induces an isomorphism
on degree $(2^k-1)\rho$ so there is an element $x \in \hf_{\star}\ev_{2^k-1}$ which hits
$\overline{\xi}_k$. So, by commutativity of the diagram, there is an element $Q^{2^k\rho}x \in
\hf_{\star}\ev$ which hits $Q^{2^k\rho}\overline{\xi}_k = \overline{\xi}_{k+1} \in P_{\star}$,
by Theorem \ref{action}, which
was to be shown.
\newline\newline
\underline{Step 3}. It suffices to show that $P^0\rpi_{n\rho}\ev$ is constant and torsion-free
for $n\le 2^{k+1}-2$.
 
We'll show that, in this case, we can build the diagram from Step 2. We construct the map
	\[
	F_{2^{k+1}}(\mathbb{P}_2(\ev_{2^k-1})) \to \ev
	\]
by inducting up the slice tower of $\ev$. Since the odd slices vanish, the obstructions live in
	\[
	[F_{2^{k+1}}(\mathbb{P}_2(\ev_{2^k-1})), \Sigma^{n\rho+1}\textup{H}P^0\rpi_{n\rho}\ev]^{C_2}.
	\]
When $n\le 2^{k+1} - 2$, we are, by hypothesis, in the realm of Lemma \ref{coh-vanish-torsion-free}
and these groups vanish. When $n\ge 2^k + 2^k - 1 = 2^{k+1} - 1$ these vanish by Lemma
\ref{easy-vanishing}. The homotopy commutativity of the diagram in Step 2 is a statement
about classes in $\hf^0$. These are detected on the bottom cell, and the result follows
since $S^0 \to \ev \to \hf$ is the unit. 
\newline\newline
\underline{Step 4}. It suffices to show that 
	\begin{enumerate}[(i)]
	\item $\underline{\textup{Ext}}^{s,V}_{\arep}(\hf_{\star}\ev) = 0$ for $V-s = (2^{k+1}-2)\rho -1$,
	\item $\underline{\textup{Ext}}^{s,V}_{\arep}(\hf_{\star}\ev)$ has surjective restriction map
	for $V-s = (2^{k+1} -2)\rho$,
	\item the \emph{underlying} group of $\underline{\textup{Ext}}^{s,V}_{\arep}(\hf_{\star}\ev)$ is
	 $v_0$-torsion-free when $V-s = (2^{k+1}-2)\rho$,
	\item the \emph{underlying} group of  $\underline{\textup{Ext}}^{s,V}_{\arep}(\hf_{\star}\ev)$ is
	zero for $V-s = (2^{k+1} - 2)\rho +1$.
	\end{enumerate}
	
By Step 3, we need to know that $P^0\rpi_{n\rho}\ev$ is constant and torsion-free in some range.
This is equivalent to checking that the underlying groups of $\rpi_{n\rho}$ are torsion-free and
the restriction map is surjective. 
Using the fact that $v_0$ corresponds to multiplication by 2,
we can check this property on the $E_{\infty}$-page of the Adams spectral sequence. The conditions
listed above ensure that the $E_2$-page is equal to the $E_{\infty}$-page in homotopy
dimension $n\rho$, modulo $a_{\sigma}$, for $n$ in the requisite range. This is all we need, since
$a_{\sigma}$ vanishes under the restriction map.
\newline\newline
\underline{Step 5} 
	\begin{enumerate}[(a)]
	\item $\underline{\textup{Ext}}^{s,V}_{\arep}(\hf_{\star}\ev) =
	\underline{\textup{Ext}}^{s,V}_{\arep}(P_{\star})$ for $V-s = (2^{k+1} -2)\rho -1$,
	\item modulo $a_{\sigma}$, we have $\underline{\textup{Ext}}^{s,V}_{\arep}(\hf_{\star}\ev) =
	\underline{\textup{Ext}}^{s,V}_{\arep}(P_{\star})$ when $V-s = (2^{k+1}-2)\rho$,
	\item the \emph{underlying} groups
	$\underline{\textup{Ext}}^{s,V}_{\arep}(\hf_{\star}\ev)$ agree with the
	\emph{underlying} groups of $\underline{\textup{Ext}}^{s,V}_{\arep}(P_{\star})$ for
	$V-s = (2^{k+1}-2)\rho$ and $V-s = (2^{k+1} -2)\rho +1$.
	\end{enumerate}

First, by our induction hypothesis, Lemma \ref{approx-split} implies that
$\hf_{\star}\ev$ is of the form $P_{\star} \otimes C$ in a range, where $C$ is primitive
on even generators. Then Lemma \ref{approx-ext} implies that, in our range,
the Ext groups in homotopy dimension $n\rho -1$ vanish. It follows that elements
of $C$ in positive degree and in our range must survive to homotopy classes, and
the same argument as in Step 1 shows that this can't happen. So $\hf_{\star}\ev$
and $P_{\star}$ are isomorphic as comodules modulo terms of high degree. So, again
by Lemma \ref{approx-ext} and Proposition \ref{ext-computation} the result follows.
\end{proof}

\begin{remark} I learned from Paul Goerss that it is possible to modify this induction to 
give a construction of $\bpr$ without using power operations. The idea is to
put a bit more control over the attaching maps in Priddy's construction, and keep track
of the comodule structure and homotopy as you go along using Lemma \ref{approx-ext}
and Proposition \ref{ext-computation}. The present approach has the advantage of
efficiency, and we simultaneously
construct and compute the action of power operations on $\bpr$, which is useful
in proving splitting theorems.
\end{remark}

\begin{theorem}\label{main-bpr-thm}
Without invoking the existence of $\bpr$, it is possible to establish the following
properties of $\textup{Even}(S^0)$ from its construction:
	\begin{enumerate}[(i)]
	\item The spectrum underlying $\textup{Even}(S^0)$ is $\textup{BP}$,
	\item the geometric fixed points of $\textup{Even}(S^0)$ are equivalent to $\textup{H}\mathbb{F}_2$,
	\item the homotopy Mackey functors $\rpi_{*\rho}\textup{Even}(S^0)$ are constant and given by
	$\underline{\mathbb{Z}}_{(2)}[v_1, v_2, ...]$ where $|v_i| = (2^i-1)\rho$, and $\rho$ is the regular
	representation of $C_2$.
	\item the homotopy Mackey functors $\rpi_{*\rho -1}\textup{Even}(S^0)$ vanish.
	\end{enumerate}
These last two properties determine the slice tower of $\textup{Even}(S^0)$.
\end{theorem}
\begin{proof} Part (iii) follows from Hu-Kriz's equivariant version of the
Adams spectral sequence together with Proposition \ref{ext-computation}.

For parts (i) and (ii), standard splitting theorems apply to show
that $\Phi^{C_2}\textup{Even}(S^0)$ splits as a wedge of suspensions of $\textup{H}\F$
and $\Phi^e\textup{Even}(S^0)$ splits as a wedge of suspensions of
$\textup{BP}$. But there is only room for one copy in each case.

Part (iv) holds by construction, and the last statement follows from
Theorem \ref{slice-tower}.
\end{proof}

At long last, we invoke the existence of $\bpr$ and its known properties to
deduce a final sanity check: we have produced the correct spectrum, in the end.

\begin{theorem}\label{its-bpr} $\ev \cong \bpr$.
\end{theorem}
\begin{proof}
A map extending the unit
exists by obstruction theory applied to the
known description of the slice tower of $\bpr$. To see this
is an equivalence, one can check on underlying spectra
and geometric fixed points. Both $\textup{BP}$ and $\textup{H}\F$
have the property that any endomorphism inducing an isomorphism
on $\pi_0$ is an equivalence. Indeed, both spectra are connective
and have cohomology which is generated
by the unit over the Steenrod
algebra. The result follows. 

Alternatively, one can check that the map induces an isomorphism of
homology, hence Adams $E_2$-terms, and hence an equivalence
on $\rpi_{*\rho}$. Since $\rpi_{*\rho -1} = 0$ for both spectra, we can
apply an $RO(C_2)$-graded version of the Whitehead theorem due
to Hill-Meier \cite[Lem 2.4]{HM}.
\end{proof}

\printbibliography
\end{document}